\documentclass[11pt,a4paper,reqno,english]{amsart}
\title[]{Uniform resolvent estimates for magnetic operators}
\usepackage{%
    amssymb%
    ,graphicx%
    ,mathrsfs%
    ,eucal%
    ,enumitem%
    ,geometry%
    ,esint}
\usepackage[%
    pdfauthor={Piero D'Ancona}%
    ,colorlinks=true%
    ,linkcolor=magenta%
    ,citecolor=cyan%
    ,urlcolor=blue%
    ,hyperfootnotes=false%
]{hyperref}

\newcommand{\bra}[1]{\langle #1 \rangle}
\newcommand{\one}[1]{\mathbf{1}_{#1}}

\newcommand{\pdfmath}[1]{\texorpdfstring{$#1$}{}}

\numberwithin{equation}{section}
\newtheorem{theorem}{Theorem}[section]
\newtheorem{corollary}[theorem]{Corollary}
\newtheorem{lemma}[theorem]{Lemma}
\newtheorem{proposition}[theorem]{Proposition}
\theoremstyle{remark}
\newtheorem{remark}[theorem]{Remark}

\theoremstyle{definition}
\newtheorem{definition}[theorem]{Definition}
\newtheorem{example}[theorem]{Example}

\date{\today}

\author[P.~D'Ancona]{Piero D'Ancona}
\address{Piero D'Ancona:
Dipartimento di Matematica\\
Sapienza Universit\`{a} di Roma\\
Piazzale A.~Moro 2\\
00185 Roma\\
Italy}
\email{dancona@mat.uniroma1.it}
\author[Z.~Yin]{Zhiqing Yin}
\address{Zhiqing Yin:
Department of Mathematics,
Beijing Institute of Technology,
Beijing 100081\\
and
Dipartimento di Matematica\\
Sapienza Universit\`{a} di Roma\\
Piazzale A.~Moro 2\\
00185 Roma\\
Italy}
\email{zhiqingyin@bit.edu.cn, zhiqing.yin@uniroma1.it}

\thanks{%
The authors are partially supported by the MIUR PRIN project 2020XB3EFL, ``Hamiltonian and Dispersive PDEs'', 
by the Progetto Ricerca Scientifica 2024
``Wave dynamics in heterogeneous media'' of Sapienza University,
and by the Gruppo Nazionale per l'Analisi Matematica, 
la Probabilit\`{a} e le loro Applicazioni (GNAMPA)}%

\begin{document}

\begin{abstract}
  We prove Kenig--Ruiz--Sogge type uniform resolvent estimates
  for selfadjoint magnetic Schr\"{o}dinger operators
  $H=(i\partial+A(x))^2+V(x)$ on $\mathbb{R}^{n}$,
  $n\ge3$. Under suitable decay assumptions on the
  electric and magnetic potentials, and excluding a threshold
  resonance at zero, we show that for all
  $z \in \mathbb{C}\setminus[0,+\infty)$,
  \begin{equation*}
    \|(H-z)^{-1}\phi\|_{L^{q}}\lesssim|z|^{\theta(p,q)}
    (1+|z|^{\gamma})
    \|\phi\|_{L^{p}}
  \end{equation*}
  throughout the full free resolvent range
  $(\frac1p,\frac1q)\in\Delta(n)$, where
  $\theta(p,q)=\frac n2(\frac1p-\frac1q)-1$.
  Here $\gamma=\frac 12\frac{n-1}{n+1}$ under the basic
  magnetic decay hypothesis, or
  $\gamma=\frac{n-1}{4n}$ under a different decay assumption
  on $A(x)$; for the second case we use a weak endpoint estimate
  of Frank--Simon type
  \begin{equation*}
    \|R_{0}(z)\phi\|
      _{L^{\frac{2n}{n-1},\infty}_{r}L^{2}_{\omega}}
    \lesssim
    |z|^{-\frac12}
    \|\phi\|_{L^{\frac{2n}{n+1},1}_{r}L^{2}_{\omega}}.
  \end{equation*}
  The result extends the
  known electromagnetic estimates from fixed frequency
  and a smaller exponent region to all frequencies and the
  full Kenig--Ruiz--Sogge range. We also prove a variant with
  weaker local assumptions in a smaller range $\Delta_1(n)$.
  As applications, we obtain $L^p-L^{p'}$ restriction type
  estimates for the density of the spectral measure of magnetic
  Schr\"{o}dinger operators, and an eigenvalue enclosure result 
  for complex scalar perturbations.
\end{abstract}

\maketitle



\section{Introduction}\label{sec:intr}

Uniform resolvent estimates are a quantitative form of the limiting
absorption principle in suitable scales of Banach spaces. 
They represent a meeting point of harmonic analysis and spectral 
theory:
the same estimates underlie unique continuation, restriction and
spectral measure bounds, spectral multiplier theorems, dispersive
estimates, and eigenvalue bounds for nonselfadjoint
perturbations.  The model case is the free resolvent on
$\mathbb{R}^{n}$, $n\ge1$,
\begin{equation*}
  R_{0}(z)=(-\Delta-z)^{-1},
\end{equation*}
which is defined as a bounded operator on $L^{2}(\mathbb{R}^{n})$
for $z\not\in\sigma(-\Delta)=[0,\infty)$;
when approaching the spectrum, the bound degenerates as
$\|R_{0}(z)\|_{L^{2}\to L^{2}}\simeq d(z,\sigma(-\Delta))^{-1}$.
The point of the Kenig--Ruiz--Sogge theory is this $L^{2}$ loss
can be removed by changing the source and target spaces
to $L^{p}$--$L^{q}$ spaces, obtaining bounds which remain
uniform up to
the continuous spectrum.  For magnetic Schr\"{o}dinger operators
of the form
$H=(i\partial+A)^{2}+V$ this question is especially delicate,
due to the presence of first order terms
$2iA\cdot\partial$.  The main purpose of this
paper is to recover, for magnetic Hamiltonians, the free
Kenig--Ruiz--Sogge exponent range as far as the first order
terms allow.

Resolvent estimates in Lebesgue spaces were first investigated in
\cite{KenigRuizSogge87-a},
\cite{KatoYajima89-a}
and completed in
\cite{CarberySoria88-a},
\cite{RuizVega93-a},
\cite{Gutierrez04-a},
\cite{FrankSimon17-a},
\cite{RenXiZhang18-a},
\cite{KwonLee20-a}.
In the basic formulation, they take the form
\begin{equation}\label{eq:URER0}
  \|R_{0}(z)\phi\|_{L^{q}}\lesssim
  |z|^{\frac 12(\frac np-\frac nq)-1}
  \|\phi\|_{L^{p}}
\end{equation}
for all $n\ge2$ and $(\frac 1p,\frac 1q)$ belonging to
the set $\Delta(n)\subset[0,1]^{2}$, defined as
\begin{equation*}
  \textstyle
  \Delta(n)=
  \left\{(\frac1p,\frac1q)\in [0,1]^2:
  \frac{2}{n+1}\leq\frac1p-\frac1q\leq\frac2n,\
  \frac1p>\frac{n+1}{2n},\ \frac1q<\frac{n-1}{2n}
  \right\}
  \setminus\{(1,0)\}.
\end{equation*}
In the picture below, $\Delta(n)$ is the quadrilateral
$ABB'A'$, including the open segments $BB'$, $AA'$ but with
the closed segments $AB$, $B'A'$ removed
(in dimension $2$ the set $\Delta(2)$ degenerates to the 
pentagon $BXZX'B'$). The estimate is valid also along the
sides $AB$ and $A'B'$, but in the weak version
$L^{p,1}\to L^{q,\infty}$.
\begin{figure}[h] 
  \centering
  \includegraphics[width=13.3cm]{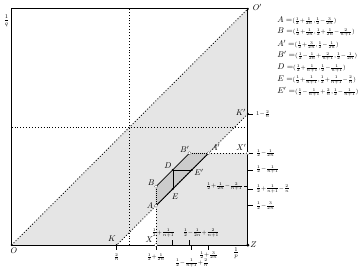} 
\end{figure}

For later use we introduce a second set
\begin{equation*}
  \textstyle
  \Delta_{1}(n)=
  \{(\frac 1p,\frac 1q)\in[0,1]^{2}:
  \frac1p-\frac1q\le \frac 2n,\quad
  \frac 1p\ge \frac 12+\frac{1}{n+1},\quad
  \frac 1q\le \frac 12-\frac{1}{n+1}
  \}=DEE'.
\end{equation*}
Note that the points $A= (\frac{1}{p_{A}},\frac{1}{q_{A}})$
and $A'= (\frac{1}{p_{A'}},\frac{1}{q_{A'}})$ are dual to
each other, i.e.~$p_{A'}=q_{A}'$, $q_{A'}=p_{A}'$;
likewise for $B,B'$ and $E,E'$.
In Section \ref{sec:free_case}, we discuss these estimates at
length and give some minor improvements at endpoints.  The
important fact is that $\Delta(n)$ is the
free exponent geometry against which the perturbative problem
should be measured.

A natural question is what part of the theory for the free 
case survives under lower order perturbations.  For scalar
electric potentials this problem is well developed.
Focusing on Kenig--Ruiz--Sogge type results,
in \cite{Mizutani20-b} estimate \eqref{eq:URER0} was extended
in dimensions $n\ge3$ to Schr\"{o}dinger operators of the form
\begin{equation*}
  H=-\Delta+V(x),
\end{equation*}
provided $V(x)$ is in the space 
$L^{n/2,\infty}_{0}(\mathbb{R}^{n})$, 
defined as the closure of
$C_{c}^{\infty}(\mathbb{R}^{n})$ in the Lorentz norm
$L^{n/2,\infty}$.  This result is a natural benchmark for
large electric perturbations: it recovers the free uniform
Sobolev range under a scale critical assumption and derives
applications to Strichartz estimates, spectral multiplier
theorems, and eigenvalue bounds. See the references in 
\cite{Mizutani20-b} for
earlier scalar potential results.

Another important
theory was developed by Ionescu and Schlag
\cite{IonescuSchlag06-a},
who prove an Agmon--Kato--Kuroda
theorem for a broad class of admissible perturbations; in
particular their class includes first order differential
perturbations of the form
$a\cdot\nabla-\nabla\cdot\overline{a}$.
The limiting absorption principle in \cite{IonescuSchlag06-a}
is formulated in a different Banach scale, adapted to scattering
and absence of singular spectrum, rather than as a
Kenig--Ruiz--Sogge $L^p-L^q$ estimate.
Nevertheless, this is a central comparison point,
showing that Stein--Tomas methods
can be used to treat first order perturbations.

There is also a substantial geometric branch of the same problem.
Uniform Sobolev estimates have been proved for nontrapping
asymptotically conic metrics \cite{GuillarmouHassell14-a}, and
on compact Riemannian manifolds the corresponding resolvent
problem is connected with spectral cluster estimates and the
distribution of eigenvalues, see for instance
\cite{DosSantosFerreiraKenigSalo14-a},
\cite{BourgainShaoSoggeYao15-a},
\cite{ShaoYao14-a},
\cite{FrankSchimmer17-a}.
These results are parallel to the Euclidean setting
considered here, but do not replace it; compactness, spectral
clusters, and metric geometry introduce a different set of
obstructions from those produced by electromagnetic first order
terms in $\mathbb{R}^{n}$.

For selfadjoint
magnetic operators
\begin{equation}\label{eq:oper}
  H=(i \partial+A(x))^{2}+V(x)
\end{equation}
with real valued electromagnetic potentials $A,V$, the
Kenig--Ruiz--Sogge $L^p-L^q$ problem is much less developed.
In this direction, the closest predecessor in
the $L^p$--$L^q$ framework is due to Garcia.  
In \cite{Garcia15-a} (see also \cite{Garcia11-a}),
the operator \eqref{eq:oper} is
considered in dimension $n\ge3$ under
the assumptions: $\partial \cdot A=0$ and
for some $\delta>0$
\begin{equation}\label{eq:AVgarc15}
  |V(x)|\lesssim
    |x|^{-\frac 32-\frac{1}{n+1}+\delta}+|x|^{-1-\delta},
  \qquad
  |A(x)|\lesssim\min\{
    |x|^{-\frac 12-\frac{1}{n+1}+\delta},
    |x|^{-1-\delta}
  \}
\end{equation}
while for the magnetic field
$B(x)=[\partial_{j}A_{k}-\partial_{k}A_{j}]$
the assumptions are
\begin{equation}\label{eq:Bgarc15}
  |x^{T}B(x)|\lesssim|x|^{-\delta}
  \quad \text{if}\ |x|\ge 1,
  \qquad
  |B(x)|\lesssim{|x|}^{-2+\delta}
  \quad \text{if}\ |x|\le 1.
\end{equation}
Then $R(z)=(H-z)^{-1}$ satisfies
\begin{equation}\label{eq:gar15}
  \textstyle
  \|R(1\pm i \epsilon)\phi\|_{L^{q}}\lesssim
  \|\phi\|_{L^{p}}
  \quad\text{provided}\quad
  (\frac 1p,\frac 1q)\in \Delta_{1}(n),
  \quad
  n\ge3.
\end{equation}
Compared with \eqref{eq:URER0}, we see that
both the frequency $z$ and the range of $p,q$
are restricted.
In dimension $n=2$, the Aharonov--Bohm operator
\begin{equation*}
  H_{AB}=(i\partial+A(x))^{2},
  \qquad
  A(x)=\alpha
  \left(-\frac{x_{2}}{|x|^{2}},\frac{x_{1}}{|x|^{2}}\right),
  \qquad
  \alpha\in \mathbb{R}
\end{equation*}
was studied in
\cite{FanelliZhangZheng23-a},
where the uniform resolvent estimate was proved
for the same range of parameters as in the free case.
Since $H_{AB}$ has the same scaling as $-\Delta$, the estimate
is valid for all frequencies $z$ as for \eqref{eq:URER0}.
See also \cite{CueninKenig17-a} for some partial results in the
case of unbounded potentials.  These results leave open the
following Euclidean magnetic analogue of the scalar potential
theory: can one prove resolvent estimates for \eqref{eq:oper}
throughout the free Kenig--Ruiz--Sogge range, 
with explicit dependence on the frequency?

Under the assumptions below, our result recovers the full
free region $\Delta(n)$ including the weak endpoint bounds,
and in this sense it is sharp at
the level of exponent geometry.  We do not claim that either the
decay and regularity hypotheses on $A,V$, or the additional
high frequency loss in the estimate, are optimal.  In particular,
the theorem below removes both the fixed frequency restriction
and the smaller exponent region in \cite{Garcia15-a}, at the
price of the additional high frequency factor in \eqref{eq:UREH}.

We stress that
uniform resolvent estimates are useful beyond the mapping
property itself, connecting elliptic resolvent analysis with
restriction type bounds for the density of the spectral measure,
spectral multiplier and dispersive estimates, and eigenvalue
bounds or spectral enclosures for nonselfadjoint perturbations;
see for instance
\cite{FrankSimon17-a,HuangYaoZheng18-a,Mizutani20-b}.  In the
present paper we record two representative consequences:
spectral measure estimates in the full range allowed by our
resolvent theorem, and a spectral enclosure for complex scalar
perturbations.

\bigskip
\textbf{Assumption (H)}.
Let $n\ge3$, $\delta>0$, $\mu>1$,
$w(x)=\bra{\log|x|}^{\mu}\bra{x}^{\delta}$.
The operator $H$ in \eqref{eq:oper}
is selfadjoint with domain $H^{2}(\mathbb{R}^{n})$, 
nonnegative, and
\begin{equation}\label{eq:VAass}
  w(x)|x|^{2}
  (V+i \partial \cdot A)
  \in L^{\infty},
  \quad
  w(x)x^{T}B\in L^{\infty},
  \quad
  w(x)|x| A\in L^{\infty}\cap\dot H^{1/2}_{2n}
\end{equation}
where
$B=[\partial_{j}A_{k}-\partial_{k}A_{j}]$
is the magnetic field. Moreover, 0 is not a resonance for $H$.

\bigskip

\begin{definition}[Resonance]\label{def:reson}
  $0$ is a \emph{resonance} for $H$ if
  $\exists v\in H^{2}_{loc}(\mathbb{R}^{n}\setminus0)
    \cap H^{1}_{loc}$,
  $v \neq0$ distributional solution of $Hv=0$, with the properties
  \begin{equation}\label{eq:resona}
    |x|^{\frac n2-2-\sigma}v\in L^{2}
    \quad\text{and}\quad
    |x|^{\frac n2-1-\sigma}\partial v\in L^{2}
    \qquad \forall \sigma\in(0,\sigma_{0}).
  \end{equation}
  for some $\sigma_{0}>0$.
  Then $v$ is called a \emph{resonant state}
  at 0 for $H$.
  (If $n\ge5$ this condition reduces to 0
  being an eigenvalue of $H$).
\end{definition}

Our main result is the following:

\begin{theorem}\label{the:2}
  Let $n\ge3$.
  Assume the operator $H=(i \partial+A(x))^{2}+V(x)$ satisfies
  \textbf{(H)}, and in addition
  \begin{equation}\label{eq:assA2}
    \textstyle
    \one{|x|\ge1}|x|^{1/2}A\in L^{s}
    \quad\text{for some}\quad
    s<\frac{2n(n+1)}{3n-1}.
  \end{equation}
  Then for all $(\frac 1p,\frac 1q)\in\Delta(n)$
  the following estimate holds, with
  $\theta(p,q)=\frac n2(\frac{1}{p}-\frac 1q)-1$:
  \begin{equation}\label{eq:UREH}
    \|(H-z)^{-1}\phi\|_{L^{q}}\lesssim
    |z|^{\theta(p,q)}
      (1+|z|^{\frac 12 \frac{n-1}{n+1}})
    \|\phi\|_{L^{p}},
    \qquad
    z\in \mathbb{C}\setminus[0,+\infty).
  \end{equation}
  Moreover, for the same values of $z$ and for
  $(\frac 1p,\frac 1q)\in AB$, we have the weak type estimate
  \begin{equation}\label{eq:UREweak}
    \|R(z)\phi\|_{L^{q,\infty}}\lesssim
    |z|^{\theta(p_{B},q)}
      (1+|z|^{\frac 12 \frac{n-1}{n+1}})
    \|\phi\|_{L^{p_{B},1}}
  \end{equation}
  and a corresponding dual estimate if
  $(\frac 1p,\frac 1q)\in A'B'$.
  If, instead of \eqref{eq:assA2}, one assumes
  \begin{equation}\label{eq:assA2FS}
    \one{|x|\ge1}|x|^{1+\frac{n-1}{2n}+\delta}A
    \in L^{\infty}
    \qquad\text{for some}\quad \delta>0,
  \end{equation}
  then the same conclusions \eqref{eq:UREH} and
  \eqref{eq:UREweak} hold with the factor
  $1+|z|^{\frac 12\frac{n-1}{n+1}}$ replaced by
  \begin{equation*}
    1+|z|^{\frac{n-1}{4n}}.
  \end{equation*}
\end{theorem}

\begin{example}[]\label{exa:assH}
  Assume $V,A$ satisfy the pointwise estimates, for some
  $\delta>0$,
  \begin{equation}\label{eq:assdec}
    |V(x)|+|\partial A|\lesssim 
    \frac{1}{|x|^{2+\delta}+|x|^{2-\delta}},
    \qquad
    |A(x)|+|\widehat{x}^{T}B(x)|\lesssim
    \frac{1}{|x|^{1+\delta}+|x|^{1-\delta}}
  \end{equation}
  where $\widehat{x}=\frac{x}{|x|}$,
  and $H\ge0$ with no resonance at 0. Then Assumption
  \textbf{(H)} is satisfied. Assume in addition that
  \begin{equation}\label{eq:addiass}
    |A(x)|\lesssim
    |x|^{-1-\frac{n-1}{2n}-\delta}
    \quad\text{for}\quad |x|\ge1.
  \end{equation}
  Then \eqref{eq:assA2FS} holds and
  Theorem \ref{the:2} applies.  Hence
  $R(z)$ satisfies \eqref{eq:UREH} for the entire set of indices
  $\Delta(n)$ and all frequencies $z$, with
  $1+|z|^{\frac 12\frac{n-1}{n+1}}$ replaced by
  $1+|z|^{\frac{n-1}{4n}}$.
\end{example}

\begin{remark}[The Frank--Simon alternative]\label{rem:FSref}
  The two alternatives \eqref{eq:assA2} and \eqref{eq:assA2FS}
  enter only in the estimate of the term
  $R_{0}a\cdot\partial R_{1}$ in the proof of
  Theorem \ref{the:2}.  In pointwise terms, \eqref{eq:assA2}
  corresponds at infinity to the model decay
  \begin{equation*}
    |A(x)|\lesssim |x|^{-\frac{2n}{n+1}-\delta},
  \end{equation*}
  while the natural threshold is $|A(x)|\lesssim|x|^{-1-}$.
  For a better treatment of this frequency localized term,
  we prove in Theorem \ref{the:FSendpointa} the following
  weak endpoint Frank--Simon estimate in mixed radial--angular
  norms
  \begin{equation*}
    \|R_{0}(z)\phi\|
      _{L^{\frac{2n}{n-1},\infty}_{r}L^{2}_{\omega}}
    \lesssim
    |z|^{-\frac12}
    \|\phi\|_{L^{\frac{2n}{n+1},1}_{r}L^{2}_{\omega}}
  \end{equation*}
  for all $n\ge2$, $z\in \mathbb{C}\setminus(0,+\infty)$,
  which may be of independent interest.
  This allows us to improve the above pointwise decay to
  \eqref{eq:assA2FS}, that is to say
  \begin{equation*}
    |A(x)|\lesssim 
    |x|^{-\frac{3n-1}{2n}-\delta},
    \qquad |x|\ge1.
  \end{equation*}
  Under this replacement no other 
  part of the proof changes, and in addition
  the high-frequency loss reduces to
  $|z|^{\frac{n-1}{4n}}$ instead of
  $|z|^{\frac 12\frac{n-1}{n+1}}$.

  If we drop the additional decay \eqref{eq:assA2}
  or \eqref{eq:assA2FS}, the same argument gives
  the estimate
  \begin{equation}\label{eq:scaleinv}
    \|R(z)\phi\|_{L^{q,\infty}}\lesssim
    |z|^{\theta(p_{B},q)}
    \|\phi\|_{L^{p_{B},1}_{r}L^{2}_{\omega}}
  \end{equation}
  which is weaker in terms of the norms,
  but scale invariant in terms of the frequency.
\end{remark}

Compared with the earlier estimate \eqref{eq:gar15} of
\cite{Garcia15-a}, Theorem \ref{the:2} improves the picture in
three directions.  First, the estimate is valid for all
frequencies $z$, with an explicit dependence on $z$.  Second, it
recovers the free Kenig--Ruiz--Sogge exponent region
$(\frac 1p,\frac 1q)\in\Delta(n)$.  Third, it includes the weak
endpoint estimates proved in \cite{RenXiZhang18-a} for the
unperturbed resolvent, hence the theorem is sharp in the sense of
the exponent range.  

The assumptions on the electromagnetic
potentials, and the necessity of an additional
high-frequency factor in \eqref{eq:UREH}
are separate questions.
It is still unclear what are the optimal decay and
regularity assumptions on $A,V$ required for the estimate to
hold. In \cite{Garcia15-a}, the decay assumption 
\eqref{eq:addiass} is replaced by a less singular behaviour 
near 0, see \eqref{eq:AVgarc15}, provided the set of indices is
restricted to $\Delta_{1}(n)=DEE'$ and the frequency
is restricted to $|z|=1$. 
We can extend this result to all frequencies as follows.

\begin{theorem}\label{the:1}
  Let $n\ge3$.
  Assume the operator $H=(i \partial+A(x))^{2}+V(x)$
  is selfadjoint nonnegative, 0 is not a resonance,
  $A,V$ satisfy \eqref{eq:assdec}, 
  and in addition, for some $\delta>0$,
  \begin{equation}\label{eq:assdec2}
    |A(x)|+|x||V(x)|\lesssim |x|^{-\frac 12-\frac{1}{n+1}+\delta}
    \quad\text{for}\quad |x|\le1.
  \end{equation}
  Then for all $(\frac 1p,\frac 1q)\in \Delta_{1}(n)$
  the following estimate holds:
  \begin{equation}\label{eq:garcia}
    \|R(z)\phi\|_{L^{q}}\lesssim
    |z|^{\theta(p,q)}(1+|z|^{-\frac 12})
    \|\phi\|_{L^{p}},
    \qquad
    z\in \mathbb{C}\setminus[0,\infty).
  \end{equation}
\end{theorem}

\begin{remark}[]\label{rem:potV1}
  In Theorem \ref{the:1} the potential $V(x)$
  decays as $|x|^{-2-\delta}$ for large $x$,
  while in the result of \cite{Garcia15-a} a slower
  decay $|V(x)|\lesssim |x|^{-1-\delta}$ is admitted.
  It is not difficult to extend our
  Theorem \ref{the:1} to cover this case and prove the 
  following estimate: for every $\epsilon_{0}>0$
  there exists a continuous function $C(\epsilon_{0},z)$
  such that 
  \begin{equation}\label{eq:garcia2}
    \textstyle
    \|\widetilde{R}(z)\phi\|_{L^{q}}\le
    C(\epsilon_{0},z)
    \|\phi\|_{L^{p}},
    \qquad
    |z|\ge \epsilon_{0},
    \quad
    |\Im z|\le1,
    \qquad
    (\frac 1p,\frac 1q)\in \Delta_{1}(n).
  \end{equation}
  In this case, the assumption that 0 is not a resonance
  can be dropped.
  We sketch a proof of this extension in 
  Section \ref{sec:rem}.
\end{remark}

As previously mentioned,
uniform resolvent estimates are useful well beyond the estimate
itself.  In the free and scalar potential settings they are tied
to unique continuation, spectral multiplier theorems,
Bochner--Riesz estimates, Strichartz estimates, smoothing and
local energy decay, and eigenvalue bounds for nonselfadjoint
perturbations; see for instance
\cite{KenigRuizSogge87-a},
\cite{HuangYaoZheng18-a},
\cite{FrankSimon17-a},
\cite{Mizutani20-b}.
In the present magnetic setting we focus on the first of these
consequences that follows directly from the resolvent theorem:
a restriction type estimate for the spectral measure
$dE_{H}(\lambda)=E'_{H}(\lambda)d \lambda$
of the selfadjoint operator $H$.
Denoting by $E'_{H}(\lambda)$ its density,
for the Euclidean Laplacian we have the bound
\begin{equation}\label{eq:restrfree}
  \textstyle
  \|E'_{-\Delta}(\lambda)\phi\|_{L^{p'}}\leq 
  C \lambda^{\theta(p,p')}\|\phi\|_{L^{p}},
  \qquad
  \frac 12+\frac{1}{n+1}\le\frac 1p\le1. 
\end{equation}
Since $E'_{-\Delta}$ can be written in terms of the
restriction operator on the sphere $R_{r}$ as
$E'_{-\Delta}(\lambda^{2})=
  (2\pi)^{-n}\lambda^{n-1}R^{*}_{\lambda} R_{\lambda}$,
by the standard $TT^{*}$ method,
estimate \eqref{eq:restrfree} is equivalent to the
Tomas--Stein restriction theorem for the sphere.
The spectral measure $E'_{-\Delta+V}(\lambda)$ was studied in
\cite{GoldbergSchlag04-a}, \cite{HuangYaoZheng18-a},
\cite{Mizutani20-b}. In Section \ref{sec:restr} we prove
an estimate of the form
\begin{equation*}
  \textstyle
  \|E'_H(\lambda)\phi\|_{L^{p'}}
  \leq C (\lambda)\|\phi\|_{L^{p}},
  \qquad
  \frac 12+\frac{1}{n+1}\le\frac 1p\le \frac 12+\frac3{2n}
\end{equation*}
for the operator \eqref{eq:oper}. To our knowledge,
this is the first result of this type for a magnetic
Schr\"{o}dinger operator.

To complete the Introduction, we show a standard application
of our estimates to the localization of eigenvalues for a
nonselfadjoint perturbation of $H$:

\begin{corollary}[Spectral enclosure]\label{cor:eigen-enclosure}
  Assume the hypotheses of Theorem \ref{the:2}, and set
  $\gamma=\frac12\frac{n-1}{n+1}$.  If \eqref{eq:assA2} is
  replaced by the Frank--Simon alternative \eqref{eq:assA2FS},
  set instead $\gamma=\frac{n-1}{4n}$.  Let
  $W\in L^{r}(\mathbb{R}^{n})$ be complex valued, with
  \begin{equation*}
    \frac n2\le r\le \frac{n+1}{2},
  \end{equation*}
  and let $H_{W}=H+W$ be the corresponding form perturbation.
  If $z\in \mathbb{C}\setminus[0,+\infty)$ is an eigenvalue of
  $H_{W}$, then
  \begin{equation}\label{eq:eigen-enclosure}
    \frac{|z|^{1-\frac n{2r}}}{1+|z|^{\gamma}}
    \lesssim
    \|W\|_{L^{r}}.
  \end{equation}
\end{corollary}

\begin{proof}
  This is the standard Birman--Schwinger argument; see, for
  instance, \cite{FrankSimon17-a,DAnconaFanelliKrejcirikSchiavone22-a}.
  Let
  \begin{equation*}
    p=\frac{2r}{r+1},
    \qquad
    p'=\frac{2r}{r-1}.
  \end{equation*}
  Then $(1/p,1/p')\in\Delta(n)$.  Write
  $W=U|W|$ with $|U|\le1$.  If $z$ is an eigenvalue of
  $H_{W}$, then $-1$ is an eigenvalue of
  \begin{equation*}
    K(z)=|W|^{1/2}(H-z)^{-1}U|W|^{1/2}
  \end{equation*}
  on $L^{2}$, and hence $1\le\|K(z)\|_{L^{2}\to L^{2}}$.
  By H\"{o}lder,
  \begin{equation*}
    \|U|W|^{1/2}f\|_{L^{p}}
    \le
    \|W\|_{L^{r}}^{1/2}\|f\|_{L^{2}},
    \qquad
    \||W|^{1/2}g\|_{L^{2}}
    \le
    \|W\|_{L^{r}}^{1/2}\|g\|_{L^{p'}}.
  \end{equation*}
  Applying \eqref{eq:UREH} with $q=p'$ gives
  \begin{equation*}
    1\lesssim
    \|W\|_{L^{r}}
    |z|^{\theta(p,p')}(1+|z|^{\gamma}).
  \end{equation*}
  Since $\theta(p,p')=\frac n{2r}-1$, this is exactly
  \eqref{eq:eigen-enclosure}.
\end{proof}

It is possible to give different and more precise conditions
on the eigenvalues. For example, using
the scale invariant estimate \eqref{eq:scaleinv} instead of
Theorem \ref{the:2}, one obtains that all eigenvalues of $H_{W}$
are confined in a disk of the form
\begin{equation*}
  |z|
  \lesssim
  \|W\|_{L^{n,1}_rL^\infty_\omega}\,
  \|W\|_{L^{n,1}}.
\end{equation*}
We shall pursue this line of research elsewhere.

Let us finally describe the mechanism of the proof.  We use two
families of free estimates.  The first is the Lebesgue space
Kenig--Ruiz--Sogge theory described above.  More precisely, the
free resolvent is decomposed as $R_{0}=R_{1}+R_{2}$, where
$R_{1}$ contains the singular Fourier restriction part and
$R_{2}$ is controlled by Sobolev embedding.  The second input is
the scale invariant Agmon--H\"{o}rmander theory in dyadic norms.
For $C_{j}=\{x\in \mathbb{R}^{n}:2^{j-1}\le|x|\le2^{j}\}$, set
\begin{equation}\label{eq:dyad1}
  \textstyle
  \|\phi\|_{\ell^{\infty}L^{2}}=
  \sup_{j\in \mathbb{Z}}\|\phi\|_{L^{2}(C_{j})},
  \qquad
  \|\phi\|_{\ell^{1}L^{2}}=
  \sum_{j\in \mathbb{Z}}\|\phi\|_{L^{2}(C_{j})}.
\end{equation}
A sharp free estimate of this kind is
\begin{equation}\label{eq:sharpAH}
  |z|^{\frac 12}
  \||x|^{-\frac12}R_{0}(z)\phi\|_{\ell^{\infty}L^{2}}+
  \||x|^{-\frac12}\partial R_{0}(z)\phi\|_{\ell^{\infty}L^{2}}
  \lesssim
  \||x|^{\frac12}\phi\|_{\ell^{1}L^{2}},
\end{equation}
with an additional surface term
$\sup_{R>0}\||x|^{-1}R_{0}(z)\phi\|_{L^{2}(|x|=R)}$
when $n\ge3$; see e.g.~\cite{CardosoPopovVodev04-a},
\cite{MarzuolaMetcalfeTataru08-a}, \cite{DAncona20}.
The perturbed resolvent is expanded around $R_{0}$ by writing
$H=-\Delta+a\cdot\partial+b$, with $a=2iA$ and
$b=i\partial\cdot A+|A|^{2}+V$.  The zero order terms are
handled by weighted dyadic bounds, while the first order terms
are controlled by combining the free restriction estimates with
the magnetic Agmon--H\"{o}rmander estimate from \cite{DAncona20}.

\section{Estimates for the free resolvent}\label{sec:free_case}

Here we collect, and marginally improve, a few known estimates
from \cite{KenigRuizSogge87-a}, \cite{RuizVega93-a},
\cite{Gutierrez04-a}, \cite{DAncona20}
for the free resolvent
$R_{0}(z)=(-\Delta-z)^{-1}$ which are needed in the following.
We summarize the available estimates for $R_{0}$ in 
Section \ref{sub:boun_full_reso}.

We shall make constant use of the scaling property of the free
resolvent
\begin{equation}\label{eq:scaR0}
  \textstyle
  R_{0}(z)=
  |z|^{-1}S_{\sqrt{|z|}}R_{0}(\frac{z}{|z|})S_{1/\sqrt{|z|}},
  \qquad
  z\in \rho(-\Delta)=
  \mathbb{C}\setminus[0,\infty)
\end{equation}
where $S_{t}u(x)=u(tx)$ denotes the scaling operator.
We fix a radial cutoff $\psi(\xi)\in C_{c}^{\infty}(\mathbb{R}^{n})$
with 
\begin{equation*}
  \one{|\xi|\le3/2}\le \psi(\xi)\le\one{|\xi|\le2}
\end{equation*}
and write $\psi^{c}=1-\psi$.
Then for $|z|=1$ we define the truncated operators
\begin{equation}\label{eq:R1R2}
  R_{1}(z)=\psi(D)R_{0}(z),\qquad
  R_{2}(z)=\psi^{c}(D)R_{0}(z).
\end{equation}
while for arbitrary $z\in \rho(-\Delta)$ we define by scaling
\begin{equation*}
  \textstyle
  R_{j}(z)=
  |z|^{-1}S_{\sqrt{|z|}}R_{j}(\frac{z}{|z|})S_{1/\sqrt{|z|}},
  \qquad
  j=1,2.
\end{equation*}
Note that if $R_{j}:L^{p,s}\to L^{q,r}$ ($j=0,1,2$) is bounded 
between any two Lorentz spaces with
norm $C$ for $|z|=1$, $z\neq1$, from the scaling relation we get
for all $z\in \rho(-\Delta)$
\begin{equation}\label{eq:scaLpLq}
  \textstyle
  \|R_{j}(z)\phi\|_{L^{q,r}}\le C
  |z|^{\theta(p,q)}\|\phi\|_{L^{p,s}}
  \qquad
  \theta(p,q)=\frac n2(\frac{1}{p}-\frac{1}{q})-1.
\end{equation}

The two pieces $R_{1},R_{2}$ satisfy different estimates;
we think that the splitting $R_{0}=R_{1}+R_{2}$
gives a more clear picture of why some estimates
are failing, and at which points.
Namely, $R_{1}$ is bounded for indices in the pentagon
$BB'X'ZX$, while $R_{2}\simeq \bra{D}^{-2}$ is bounded
in the region $OO'HK$, essentially equivalent to a Sobolev
embedding with a loss of $2$ derivatives.
Uniform resolvent estimates for $R_{0}$ are possible only 
where the two areas overlap, as represented by 
the darker area in the picture.

\subsection{Estimates for \pdfmath{R_{1}}}\label{sec:estiR1}

We begin by a simple Bernstein type estimate in Lorentz 
spaces for a multiplier operator $\chi(D)$ with a 
well behaved symbol.

\begin{lemma}\label{lem:bernslor}
  Assume 
  $\widehat{\chi}(\xi)\in L^{1}\cap L^{\infty}(\mathbb{R}^{n})$.
  Then $\chi(D)$ satisfies
  \begin{equation}\label{eq:bernst}
    \|\chi(D)\phi\|_{L^{q,r}}\lesssim
    \|\widehat{\chi}\|_{L^{1}\cap L^{\infty}}
    \|\phi\|_{L^{p,s}}
  \end{equation}
  in the following cases:
  \begin{enumerate}[label=(\roman*)]
    \item $1<q<p<\infty$, $r,s\in(0,\infty]$
    \item $p=q\in(1,\infty)$, $r=s\in(0,\infty]$
    or $p=q=r=s=\infty$
    or $p=q=r=s=1$
    \item $q=r=\infty$, $p\in(1,\infty)$, $s\in(0,\infty]$
    \item $p=s=1$, $q\in(1,\infty)$, $r\in(0,\infty]$
  \end{enumerate}
\end{lemma}

\begin{proof}
  We can write $\chi(D)\phi=\check \chi* \phi$ as a convolution
  with the inverse Fourier transform of $\chi$; by Young's
  inequality this gives for all $p\in[1,\infty]$
  \begin{equation*}
    \|\chi(D)\phi\|_{L^{p}}\le
    \|\widehat{\chi}\|_{L^{1}}\|\phi\|_{L^{p}},
    \qquad
    \|\chi(D)\phi\|_{L^{\infty}}\le
    \|\widehat{\chi}\|_{L^{\infty}}\|\phi\|_{L^{1}}
  \end{equation*}
  and by complex interpolation we get
  $\|\chi(D)\phi\|_{L^{q}}\le
    \|\widehat{\chi}\|_{L^{1}\cap L^{\infty}}\|\phi\|_{L^{p}}$
  for all $1\le p\le q\le \infty$.
  Keeping $p$ fixed and applying real interpolation between
  $L^{p}\to L^{q_{0}}$, $L^{p}\to L^{q_{1}}$ for some
  $1\le p\le q_{0}< q_{1}\le\infty$ we get boundedness
  $L^{p}\to L^{q,r}$ for all $1\le p<q<\infty$ and $r\in(0,\infty]$;
  then interpolating
  $L^{p_{0}}\to L^{q,r}$ and $L^{p_{1}}\to L^{q,r}$ 
  for some $1\le p_{0}<p_{1}<q<\infty$ we get (i).
  To prove (ii) we interpolate between
  $L^{p_{0}}\to L^{p_{0}}$ and
  $L^{p_{1}}\to L^{p_{1}}$ or apply directly the
  standard estimates with $p_{0}=1$ and $p_{0}=\infty$.
  Case (iii) follows interpolating
  $L^{p_{0}}\to L^{\infty}$ and
  $L^{p_{1}}\to L^{\infty}$ for arbitrary $p_{j}$,
  and case (iv) interpolating
  $L^{1}\to L^{q_{0}}$ and
  $L^{1}\to L^{q_{1}}$ for arbitrary $q_{j}$.
\end{proof}

Note that $R_{1}(z)$ has symbol $(|\xi|^{2}-z)^{-1}\psi(\xi)$,
which is uniformly bounded with all derivatives provided
$|z|=1$ and $dist(z,\sigma(-\Delta))\ge c>0$;
thus $R_{1}(z)$ satisfies all the
estimates of the previous Lemma, and the problem is only
to show that the operator norm remains bounded as $z$ approaches
the positive real axis.
In the following estimates, the power of $|z|$ is always
$\frac{|\alpha|}{2}+\theta(p,q)$ as dictated by scaling.

\begin{theorem}\label{the:estR1}
  For $n\ge2$, $\alpha\in \mathbb{N}^{n}_{0}$ the following
  estimates hold.
  At point $B$ we have:
  \begin{equation}\label{eq:estR1aweak}
    \textstyle
    \|\partial^{\alpha}R_{1}(z)\phi\|_{L^{q_{B},\infty}}
    \lesssim
    |z|^{\frac{|\alpha|}{2}-\frac{1}{n+1}}
      \|\phi\|_{L^{p_{B},1}},
    \quad
    \frac 1{p_{B}}= \frac 12+\frac{1}{2n},
    \quad
    \frac 1{q_{B}}=\frac 12+\frac{1}{2n}-\frac{2}{n+1}.
  \end{equation}
  and a similar $L^{q_{B}',1}\to L^{p_{B}',\infty}$ estimate
  at the dual point $B'$. At point $X$ we have:
  \begin{equation}\label{eq:estR1aweakX}
    \textstyle
    \|\partial^{\alpha}R_{1}(z)\phi\|_{L^{\infty}}
    \lesssim
    |z|^{\frac{|\alpha|}{2}+\frac{n-3}{4}}\|\phi\|_{L^{p_{B},1}}
  \end{equation}
  and at $X'$ we have the dual estimate
  $L^{1}\to L^{p_{B}',\infty}$.
  Along the open line $BX$ we have
  \begin{equation}\label{eq:estR1aweakBX}
    \textstyle
    \|\partial^{\alpha}R_{1}(z)\phi\|_{L^{q,1}}
    \lesssim
    |z|^{\frac{|\alpha|}{2}+\frac{n-3}{4}-\frac n{2q}}
    \|\phi\|_{L^{p_{B},1}},
    \qquad
    0<\frac 1q< \frac{1}{q_{B}}
  \end{equation}
  and the dual estimate $L^{q',\infty}\to L^{p_{B}',\infty}$
  along the open line $B'X'$,
  while at $Z$ we have
  \begin{equation}\label{eq:estR1Z}
    \|\partial^{\alpha}R_{1}(z)\phi\|_{L^{\infty}}
    \lesssim
    |z|^{\frac{|\alpha|}{2}+\frac{n}{2}-1}
    \|\phi\|_{L^{1}}.
  \end{equation}
  Along the open line $BB'$ we have for all $r\in(0,\infty]$
  \begin{equation}\label{eq:estR1BBp}
    \textstyle
    \|\partial^{\alpha}R_{1}(z)\phi\|_{L^{q,r}}
    \lesssim
    |z|^{\frac{|\alpha|}{2}-\frac{1}{n+1}}
    \|\phi\|_{L^{p,r}},
    \qquad
    \frac 1p-\frac 1q=\frac{2}{n+1},
    \quad
    \frac 12-\frac{1}{2n}<\frac 1p<
    \frac 12-\frac{1}{2n}+\frac{2}{n+1}.
  \end{equation}
  Finally we have
  \begin{equation}\label{eq:estR1a}
    \textstyle
    \|\partial^{\alpha}R_{1}(z)\phi\|_{L^{q,1}}
    \lesssim
    |z|^{\frac{|\alpha|}{2}+\theta(p,q)}\|\phi\|_{L^{p,\infty}}
  \end{equation}
  in the open pentagon $BB'X'ZX$, i.e.
  provided $p,q\in(1,\infty)$ satisfy
  \begin{equation}\label{eq:estR1apq}
    \textstyle
    \frac 1p-\frac 1q> \frac{2}{n+1},
    \qquad
    \frac 1p>\frac{1}{p_{B}}=\frac 12+\frac{1}{2n},
    \qquad
    \frac 1q< \frac{1}{p_{B}'}=\frac 12-\frac{1}{2n}.
  \end{equation}
\end{theorem}

\begin{proof}
  Since all norms are scaling invariant, by the scaling
  argument \eqref{eq:scaLpLq} it is sufficient to prove the
  claims for $|z|=1$. Moreover, recalling \eqref{eq:R1R2},
  we can pick a test function $\chi_{1}$ with
  $\chi_{1}\psi=\psi$ and write
  \begin{equation}\label{eq:decomppaR}
    \partial^{\alpha}R_{1}(z)=
    \partial^{\alpha}\chi_{1}(D)\psi(D) R_{0}(z)=
    \chi(D)R_{1}(z),
    \qquad
    \chi(\xi):=(i\xi)^{\alpha}\chi_{1}(\xi).
  \end{equation}
  Thus by Lemma \ref{lem:bernslor} we see that it is sufficient
  to consider the case $\alpha=0$ in the proof. Now, from 
  formula (40) in \cite{Gutierrez04-a} we get the estimate
  \begin{equation}\label{eq:gutR1}
    \|R_{0}(1+i \epsilon)\phi\|_{L^{p_{B}',\infty}}
    \lesssim
    \|\phi\|_{L^{q_{B}',1}}.
  \end{equation}
  which is the estimate at point $B'$.
  Since $R_{1}=\psi(D)R_{0}$, by Lemma \ref{lem:bernslor}
  the same bound is satisfied by $R_{1}(1+i \epsilon)$,
  and by the previous elementary argument it is satisfied by
  $R_{1}(z)$ for all $|z|=1$.
  This proves claim \eqref{eq:estR1aweak} so that
  point $B'$ is covered, and point $B$ follows by duality.

  The other claims follow by interpolation, duality and
  Lemma \ref{lem:bernslor}.
  Writing $R_{1}(z)=\chi_{1}(D)R_{1}(z)$ and
  using Lemma \ref{lem:bernslor}--(iii) we get
  \eqref{eq:estR1aweakX} at point $X$, while using 
  (i) and writing
  \begin{equation*}
    \|R_{1}\phi\|_{L^{q,1}}=
    \|\chi_{1}(D)R_{1}\phi\|_{L^{q,1}}\lesssim
    \|R_{1}\phi\|_{L^{q_{B},\infty}}\lesssim
    \|\phi\|_{L^{p_{B},1}}
  \end{equation*}
  we obtain \eqref{eq:estR1aweakBX} along the line $BX$.
  By duality this gives the estimate along $B'X'$:
  \begin{equation*}
    \textstyle
    \|R_{1}\phi\|_{L^{p_{B}',\infty}}\lesssim
    \|\phi\|_{L^{q',\infty}},
    \qquad
    0<\frac 1q<\frac{1}{q_{B}}.
  \end{equation*}
  Interpolating between the points $B$, $B'$ we get
  \eqref{eq:estR1BBp}.
  Finally, consider the open pentagon $BB'X'ZX$.
  By real interpolation between $Z$ and any point of the
  boundary $XBB'X'$ we get an
  $L^{p,r}\to L^{q,r}$ estimate, for arbitrary $r$.
  Then by interpolating between two estimates
  $L^{p_{0},r}\to L^{q,r}$ and $L^{p_{1},r}\to L^{q,r}$
  with the same $q$ we get the
  $L^{p,s}\to L^{q,r}$ estimate \eqref{eq:estR1apq}.
\end{proof}

For the second estimate we need a weighted version of
\eqref{eq:bernst}.

\begin{lemma}\label{lem:weighber}
  Let $\chi(\xi)\in \mathscr{S}$ and
  $a>0$, $1\le p\le q<\frac na$. Then we have
  \begin{equation}\label{eq:weighber}
    \||x|^{-a}\chi(D)\phi\|_{L^{q}}\lesssim
    \||x|^{-a}\phi\|_{L^{p}}.
  \end{equation}
  Moreover $|x|^{-a}\chi(D)|x|^{a}:L^{p}\to L^{q}(\Omega)$ is compact
  provided $q>1$ and $\Omega$ is bounded.
\end{lemma}

\begin{proof}
  Estimate \eqref{eq:weighber} is equivalent to the boundedness
  $T:L^{p}\to L^{q}$ of the operator
  $T \phi=|x|^{-a}\chi(D)|x|^{a}\phi$,
  which is an integral operator
  \begin{equation*}
    \textstyle
    T \phi(x)=\int K(x,y) \phi(y)dy,
    \qquad
    K(x,y)=\frac{|y|^{a}}{|x|^{a}}\check \chi(x-y)
  \end{equation*}
  where $\check\chi\in \mathscr{S}$ so that
  $|\check\chi(x)|\lesssim_{N} \bra{x}^{-N}$ for all $N$.
  We split the kernel $K$ as
  \begin{equation*}
    K_{1}(x,y)=K(x,y)\one{|y|\le2|x|},\qquad
    K_{2}(x,y)=K(x,y)\one{|y|>2|x|}
  \end{equation*}
  and call $T=T_{1}+T_{2}$ the corresponding splitting of $T$.
  Since $|K_{1}(x,y)|\lesssim\bra{x-y}^{-N}$ for all
  integer $N$, the operator $T_{1}$ is bounded
  $L^{p}\to L^{q}$ for all $1\le p\le q\le \infty$.
  For the second kernel $K_{2}$, since $|y|\ge 2|x|$
  we have $\bra{x-y}\simeq \bra{y} \gtrsim \bra{x}$, so that
  \begin{equation*}
    \textstyle
    |K_{2}(x,y)|\lesssim
    \frac{|y|^{a}}{|x|^{a}\bra{x}^{N}\bra{y}^{N+a}}\lesssim
      \frac{1}{|x|^{a}\bra{x}^{N}}\frac{1}{\bra{y}^{N}}.
  \end{equation*}
  This implies 
  $|T_{2}\phi(x)|\lesssim \frac{1}{|x|^{a}\bra{x}^{N}}
    \int \frac{|\phi(y)|}{\bra{y}^{N}}dy$
  so that
  \begin{equation*}
    \textstyle
    |T_{2}\phi(x)|\lesssim 
      \frac{\|\phi\|_{L^{1}}}{|x|^{a}\bra{x}^{N}}
    \qquad\text{and}\qquad 
    |T_{2}\phi(x)|\lesssim 
      \frac{\|\phi\|_{L^{\infty}}}{|x|^{a}\bra{x}^{N}}
  \end{equation*}
  and as a consequence $T_{2}:L^{1}\to L^{q}$ and
  $T_{2}:L^{\infty}\to L^{q}$ for all $1\le q<\frac na$.
  Combining the two estimates we get the claim.
  Compactness follows from the remark that
  \begin{equation*}
    \textstyle
    \one{\Omega}(x)|K(x,y)|\lesssim
    \one{\Omega}(x)\frac{|y|^{a}}{|x|^{a}}\bra{x-y}^{-N}\lesssim
    \frac{\one{\Omega}(x)}{|x|^{a}\bra{y}^{N-a}}
    \in L^{p}_{x}L^{q'}_{y}
  \end{equation*}
  since $p<\frac na$, from the general properties of
  Hille--Tamarkin operators.
\end{proof}

We next prove a version of the previous weighted estimates
in the dyadic norms \eqref{eq:dyad1} and more generally
\begin{equation}\label{eq:dyad2}
  \textstyle
  \|v\|_{\ell^{p}(2^{-js})L^{q}}=
  \left\|2^{-js}\|v\|_{L^{q}(C_{j})}
  \right\|_{\ell^{p}_{j}}
  =
  \left(\sum_{j\in \mathbb{Z}}2^{-pjs}
  \|v\|_{L^{2}(C_{j})}^{p}\right)^{1/p},
  \qquad
  p\in[1,\infty)
\end{equation}
where $C_{j}=\{x:2^{j}\le|x|<2^{j+1}\}\subseteq \mathbb{R}^{n}$.

\begin{lemma}[]\label{lem:berndyad}
  Let $\chi(\xi)\in \mathscr{S}$. For all
  $a>0$, $r\in(0,\infty]$, $1\le p\le q <\frac na$ we have
  \begin{equation}\label{eq:berndyad}
    \||x|^{-a}\chi(D)\phi\|_{\ell^{r}L^{q}}\lesssim
    \||x|^{-a}\phi\|_{\ell^{r}L^{p}}.
  \end{equation}
  Moreover 
  $|x|^{-a}\one{\Omega}\chi(D)|x|^{a}:
    \ell^{r}L^{p}\to \ell^{r}L^{q}$ is
  compact if $q>1$ and $\Omega$ is bounded.
\end{lemma}

\begin{proof}
  We recall a real interpolation formula from
  \cite{BerghLofstrom76-a} 
  (a special case of Theorem 5.6.1):
  if $q_{0},q_{1},r\in(0,\infty]$, $p\in[1,\infty]$,
  $\theta\in(0,1)$, $a_{0}\neq a_{1}\in \mathbb{R}$, then
  \begin{equation*}
    (\ell^{q_{0}}(2^{-ja_{0}})L^{p},
    \ell^{q_{1}}(2^{-ja_{1}})L^{p})_{\theta , r}
    \simeq
    \ell^{r}(2^{-ja})L^{p},
    \qquad
    a=(1-\theta)a_{0}+\theta a_{1}.
  \end{equation*}
  Estimate \eqref{eq:weighber} can be written as
  \begin{equation*}
    \textstyle
    \|\chi(D)\phi\|_{\ell^{q}(2^{-ja})L^{q}}\lesssim
    \|\phi\|_{\ell^{p}(2^{-ja})L^{p}},
    \qquad
    1\le p\le q<\frac na,\quad a>0.
  \end{equation*}
  We write the estimate at two different points
  \begin{equation*}
    \|\chi(D)\phi\|_{\ell^{q}(2^{-j(a\pm \epsilon)})L^{q}}\lesssim
    \|\phi\|_{\ell^{p}(2^{-j(a\pm \epsilon)})L^{p}}
  \end{equation*}
  and we apply the interpolation formula with 
  $q_{0}=q_{1}=q$ ($q_{0}=q_{1}=p$ at the right hand side),
  $\theta=1/2$, $r\in(0,\infty]$. We obtain
  \begin{equation*}
    \|\chi(D)\phi\|_{\ell^{r}(2^{-ja})L^{q}}\lesssim
    \|\phi\|_{\ell^{r}(2^{-ja})L^{p}}
  \end{equation*}
  which is equivalent to \eqref{eq:berndyad}.
  The final claim follows since interpolation of compact operators
  produces compact operators.
\end{proof}

\begin{proposition}[]\label{pro:gutierr}
  For $n\ge2$, $\alpha\in \mathbb{N}^{n}_{0}$ we have,
  with $\theta(p,q)=\frac n2(\frac 1p-\frac 1q)-1$:
  \begin{equation}\label{eq:estR1gut}
    \textstyle
    \||x|^{-1/2}\partial^{\alpha}R_{1}\phi\|_{\ell^{\infty}L^{q}}
    \lesssim
    |z|^{\frac{|\alpha|}{2}+\frac 14 +\theta(p,q)}
    \|\phi\|_{L^{p}},
    \qquad
    \frac{1}{n+1}+\frac 12\le \frac 1p<1,
    \quad
    q\in[2,\infty].
  \end{equation}
\end{proposition}

\begin{proof}
  By scaling we can assume $|z|=1$.
  Moreover, it is sufficient to prove that
  \begin{equation}\label{eq:estR1gut2}
    \textstyle
    \||x|^{-1/2}R_{1}\phi\|_{\ell^{\infty}L^{2}}
    \lesssim
    \|\phi\|_{L^{p}},
    \qquad
    1>\frac 1p\ge \frac{1}{n+1}+\frac 12
  \end{equation}
  and apply Lemma \ref{lem:berndyad}, since we have
  $\partial^{\alpha}R_{1}(z)=\chi(D)R_{1}(x)$
  as in \eqref{eq:decomppaR}.
  To prove estimate \eqref{eq:estR1gut2} we proceed
  exactly as in the proof of Theorem 3.1 from
  \cite{RuizVega93-a} (actually, as in the estimate of the term
  $u_{3}$ at the end of the proof; see also 
  Theorem 7 in \cite{Gutierrez04-a}).
\end{proof}

\subsection{An endpoint Frank--Simon estimate}\label{sub:endp_FS}

We shall also need the following
refinement of the mixed radial--angular estimate of
Frank and Simon \cite{FrankSimon17-a}.
This can be expressed in terms of the mixed norms
\begin{equation*}
  \|\phi\|_{L^{p,s}_{r}L^{2}_{\omega}}=
  \left\|\|\phi(r\cdot)\|_{L^{2}(\mathbb{S}^{n-1})}
  \right\|_{L^{p,s}((0,\infty),r^{n-1}dr)},
  \qquad
  s\in(0,\infty]
\end{equation*}
which are Lorentz norms in the radial direction and ordinary
$L^{2}$ norms on spheres.
Note that $\frac{2n}{n+1}=p_{B}$, $\frac{2n}{n-1}=p_{B}'$.

\begin{theorem}[Endpoint Frank--Simon estimate]\label{the:FSendpointa}
  Let $n\ge2$. Then we have
  \begin{equation}\label{eq:FSendpointRzero}
    \|R_{0}(z)\phi\|
      _{L^{\frac{2n}{n-1},\infty}_{r}L^{2}_{\omega}}
    \lesssim
    |z|^{-\frac12}
    \|\phi\|_{L^{\frac{2n}{n+1},1}_{r}L^{2}_{\omega}},
    \qquad
    z\in\mathbb{C}\setminus[0,+\infty).
  \end{equation}
\end{theorem}

\begin{proof}
  By standard arguments, it is sufficient to prove
  the estimate for $|z|=1$ and actually for $R_{0}(1+i0)$.
  In spherical harmonics decomposition, we can write
  \begin{equation*}
    R_{0}(1+i0)
    =
    \bigoplus_{\ell\ge0}
    \bigl(T_{\mu_{\ell}}\otimes I_{\mathcal K_{\ell}}\bigr)
    =
    \sum_{\ell\ge0}T_{\mu_{\ell}}\otimes P_{\ell}
  \end{equation*}
  where 
  \begin{equation*}
    \mathcal K_{\ell}
    =
    \ker\bigl(-\Delta_{\mathbb S^{n-1}}
    -\ell(\ell+n-2)\bigr),
    \qquad
    \mu_{\ell}=\ell+\frac{n-2}{2}
  \end{equation*}
  and $P_{\ell}$ is the orthogonal projection from
  $L^{2}(\mathbb{S}^{n-1})$ onto $\mathcal K_{\ell}$,
  while $T_{\mu}$ are the integral operators
  (up to a dimensional constant)
  \begin{equation*}
    T_{\mu}f(r)
    =
    A_{\mu}(r)\int_r^\infty B_{\mu}(s)f(s) s^{n-1}ds
    +
    B_{\mu}(r)\int_0^r A_{\mu}(s)f(s) s^{n-1}ds,
  \end{equation*}
  \begin{equation*}
    A_{\mu}(r)=r^{-\frac{n-2}{2}}J_{\mu}(r),
    \qquad
    B_{\mu}(r)=r^{-\frac{n-2}{2}}H_{\mu}^{(1)}(r).
  \end{equation*}
  By Lemma \ref{lem:hilbert} in the Appendix, we are reduced
  to proving the diagonal estimates
  \begin{equation*}
    \|T_{\mu}f\|_{L^{\frac{2n}{n-1},\infty}(r^{n-1}dr)}\le C
    \|f\|_{L^{\frac{2n}{n+1},1}(r^{n-1}dr)}
  \end{equation*}
  with a constant $C$ independent of $\mu\ge(n-2)/2$,
  and this is exactly the content of Lemma \ref{lem:weak-scalar},
  also in the Appendix.
\end{proof}

We recorded estimate \eqref{eq:FSendpointRzero} in view of its
intrinsic interest, however what we actually need here is
an estimate for the low frequency part $R_{1}$:

\begin{lemma}[Low frequency Frank--Simon estimate]
  \label{lem:FSendpoint}
  Let $n\ge2$. Then for all $\alpha\in \mathbb{N}^{n}_{0}$
  \begin{equation}\label{eq:FSendpoint}
    \|\partial^{\alpha}R_{1}(z)\phi\|
      _{L^{\frac{2n}{n-1},\infty}_{r}L^{2}_{\omega}}
    \lesssim
    |z|^{\frac{|\alpha|}{2}-\frac12}
    \|\phi\|_{L^{\frac{2n}{n+1},1}_{r}L^{2}_{\omega}},
    \qquad
    z\in \mathbb{C}\setminus[0,+\infty).
  \end{equation}
\end{lemma}

\begin{proof}
  If $m(\xi)$ is a smooth compactly supported multiplier
  which is a finite sum of terms
  \begin{equation*}
    m_{k}(|\xi|)Y_{k}(\xi/|\xi|)
  \end{equation*}
  with $m_{k}\in C_{c}^{\infty}((0,\infty))$ and $Y_k$ a
  spherical harmonic, then
  \begin{equation}\label{eq:mixed-smooth-mult}
    \|m(D)u\|_{L^{a,s}_{r}L^{2}_{\omega}}
    \lesssim
    \|u\|_{L^{a,s}_{r}L^{2}_{\omega}},
    \qquad 1<a<\infty,\quad 0<s\le\infty .
  \end{equation}
  Indeed, decomposing in spherical harmonics, multiplication by
  each $Y_k(\xi/|\xi|)$ couples only finitely many neighboring
  angular momenta, with coefficients uniformly bounded on
  $L^{2}(\mathbb S^{n-1})$.  The radial factors are smooth
  compactly supported Hankel multipliers, hence satisfy the
  usual Mikhlin--H\"ormander bounds on
  $L^{a}((0,\infty),r^{n-1}dr)$, uniformly in the angular
  momentum.  This gives the strong mixed estimates
  $L^{a}_{r}L^{2}_{\omega}\to L^{a}_{r}L^{2}_{\omega}$; the
  Lorentz version \eqref{eq:mixed-smooth-mult} follows by real
  interpolation between two nearby strong exponents.

  We prove the estimate at unit frequency first.  Let
  $|z|=1$.  By the definition of $R_1$,
  \begin{equation*}
    \partial^\alpha R_1(z)
    =
    m_\alpha(D)R_0(z),
    \qquad
    m_\alpha(\xi)=(i\xi)^\alpha\psi(\xi).
  \end{equation*}
  The multiplier $m_\alpha$ is of the form covered by
  \eqref{eq:mixed-smooth-mult}.  Hence, using
  \eqref{eq:mixed-smooth-mult} and
  \eqref{eq:FSendpointRzero},
  \begin{equation*}
    \|\partial^\alpha R_1(z)\phi\|
      _{L^{\frac{2n}{n-1},\infty}_{r}L^2_\omega}
    \lesssim
    \|R_0(z)\phi\|_{L^{\frac{2n}{n-1},\infty}_{r}L^2_\omega}
    \lesssim
    \|\phi\|_{L^{\frac{2n}{n+1},1}_{r}L^2_\omega}.
  \end{equation*}
  The general dependence on $|z|$ follows from scaling
  as usual.
\end{proof}

In order to handle the operators $A \cdot \partial$, we
shall need a space localized version of the previous
estimate. Recall that $p_{B}=\frac{2n}{n+1}$.

\begin{lemma}[Localized Frank--Simon estimate]
  \label{lem:FSendpointloc}
  Let $n\ge2$, $|\alpha|=1$, and let
  $\sigma_{n}=\frac{n-1}{2n}$. If
  $C_{j}=\{2^{j}\le|x|<2^{j+1}\}$ and $j\in\mathbb{Z}$, then
  \begin{equation}\label{eq:FSendpointloc}
    \|\one{C_{j}}\partial^{\alpha}R_{1}(z)\phi\|
      _{L^{p_{B}',\infty}_{r}L^{2}_{\omega}}
    \lesssim
    (1+|z|^{1/2}2^{j})^{\sigma_{n}}\|\phi\|_{L^{p_{B},1}}.
  \end{equation}
\end{lemma}

\begin{proof}
  By the rescaling \eqref{eq:scaR0}, the change of variable
  $y=\sqrt{|z|}\,x$ maps $C_{j}$ to a portion of an annulus of
  radius $|z|^{1/2}2^{j}$, contained in at most two consecutive
  dyadic shells $C_{j'}$ with $2^{j'}\sim|z|^{1/2}2^{j}$. The
  scaling factor $|z|^{-1/2}$ from $\partial R_{1}(z)$, together
  with the Jacobians
  $|z|^{-n/(2p_{B}')}$ on the left side and $|z|^{n/(2p_{B})}$
  on the right side, combine to
  $|z|^{-1/2+(n/2)(1/p_{B}-1/p_{B}')}=1$ since
  $1/p_{B}-1/p_{B}'=1/n$. Hence it is enough to prove that for
  $|z|=1$
  \begin{equation}\label{eq:FSendpointloc1}
    \|\one{C_{j}}\partial^{\alpha}R_{1}(z)\phi\|
      _{L^{p_{B}',\infty}_{r}L^{2}_{\omega}}
    \lesssim
    (1+2^{j})^{\sigma_{n}}\|\phi\|_{L^{p_{B},1}}.
  \end{equation}

  Let $\Pi_{N}$ be a smooth angular projector to spherical
  harmonics of degree $\lesssim N$. We shall use the angular
  Bernstein bound
  \begin{equation}\label{eq:angbern}
    \|\Pi_{N}\phi\|_{L^{p_{B},1}_{r}L^{2}_{\omega}}
    \lesssim N^{\sigma_{n}}\|\phi\|_{L^{p_{B},1}}.
  \end{equation}
  Indeed, if $p_0\in(1,p_B)$, the usual Bernstein inequality on
  $\mathbb S^{n-1}$ gives, for every $r>0$,
  \begin{equation*}
    \|\Pi_{N}\phi(r\cdot)\|_{L^{2}(\mathbb{S}^{n-1})}
    \lesssim N^{(n-1)(\frac{1}{p_{0}}-\frac 12)}
    \|\phi(r\cdot)\|_{L^{p_{0}}(\mathbb{S}^{n-1})}.
  \end{equation*}
  Taking the $L^{p_{0}}((0,\infty),r^{n-1}dr)$ norm of both
  sides and using
  $\|\phi\|_{L^{p_{0}}_{r}L^{p_{0}}_{\omega}}=\|\phi\|_{L^{p_{0}}}$,
  \begin{equation*}
    \|\Pi_{N}\phi\|_{L^{p_{0}}_{r}L^{2}_{\omega}}
    \lesssim N^{(n-1)(\frac 1{p_{0}}-\frac 12)}
    \|\phi\|_{L^{p_{0}}}.
  \end{equation*}
  At $p=2$, trivially
  $\|\Pi_{N}\phi\|_{L^{2}_{r}L^{2}_{\omega}}=
  \|\Pi_{N}\phi\|_{L^{2}}\lesssim\|\phi\|_{L^{2}}$.
  Since the target factor $L^{2}_{\omega}$ is fixed, these are
  bounds between scalar Lorentz spaces with values in the
  Hilbert space $L^{2}(\mathbb{S}^{n-1})$, and real interpolation
  applies. With $\theta\in(0,1)$ defined by
  $1/p_{B}=(1-\theta)/p_{0}+\theta/2$, we obtain
  \begin{equation*}
    \|\Pi_{N}\phi\|_{L^{p_{B},1}_{r}L^{2}_{\omega}}
    \lesssim N^{(1-\theta)(n-1)(\frac 1{p_{0}}-\frac 12)}
    \|\phi\|_{L^{p_{B},1}}=
    N^{(n-1)(\frac 1{p_{B}}-\frac 12)}\|\phi\|_{L^{p_{B},1}}=
    N^{\sigma_{n}}\|\phi\|_{L^{p_{B},1}},
  \end{equation*}
  which proves \eqref{eq:angbern}.

  We also use the following high angular tail estimate. If
  $P_L$ is a smooth dyadic angular projector, $L$ ranges over
  dyadic integers, and $L\ge C(1+2^j)$, then
  \begin{equation}\label{eq:highangtail}
    \|\one{C_j}\partial^\alpha R_1(z)P_Lu\|
      _{L^{p_B',\infty}_{r}L^2_\omega}
    \lesssim
    e^{-cL}\|P_Lu\|_{L^{p_B,1}_{r}L^2_\omega}.
  \end{equation}
  To prove this, denote by $K_L(r,r')$ the angular operator
  kernel of $\one{C_j}\partial^\alpha R_1(z)P_L$; thus
  $K_L(r,r')\in\mathcal L(L^2(\mathbb S^{n-1}))$. We claim that
  \begin{equation*}
    \sup_{r\in C_j}
    \left\|
      \|K_L(r,\cdot)\|_{\mathcal L(L^2_\omega)}
    \right\|_{L^{p_B',\infty}(r'^{n-1}dr')}
    \lesssim
    e^{-cL}.
  \end{equation*}
  Indeed, $R_1(z)$ is radial, hence diagonal in angular
  momentum; on the $\ell$-th sector it contributes the usual
  radial resolvent kernel, with the cutoff $\psi$. The factor
  $\xi^\alpha$ couples only $\ell$ to $\ell'$ with
  $|\ell'-\ell|\le1$. Thus, on the range of $P_L$, the output
  Bessel order is $\nu_{\ell'}\simeq L$. Since $\rho\le2$ on the
  support of $\psi$ and $r\in C_j$, the assumption
  $L\ge C(1+2^j)$ gives $r\rho\le\nu_{\ell'}/4$; for instance,
  if $P_L$ is supported where $L/2\le\ell\le2L$ we can take
  $C=64$. Therefore the large-order bound for $J_{\nu}(r\rho)$
  (that is $|J_{\nu}(t)|+|J'_{\nu}(t)|\lesssim e^{-c \nu}$
  for $0\le t\le \nu/4$) gives a factor $e^{-cL}$. After this
  factor is pulled out, the remaining bound is given
  by the same single-Bessel Lorentz estimates as in the proof
  of Lemma \ref{lem:weak-scalar}; the cutoff $\psi$ and the
  bounded angular couplings are harmless.

  Put $d\nu(r')=r'^{n-1}dr'$. For each fixed $r\in C_j$,
  the operator norm in the angular variable gives
  \begin{equation*}
    \textstyle
    \|\partial^\alpha R_1(z)P_Lu(r\cdot)\|_{L^2_\omega}
    \le
    \int_0^\infty
      \|K_L(r,r')\|_{\mathcal L(L^2_\omega)}
      \|P_Lu(r'\cdot)\|_{L^2_\omega}\,d\nu(r')
    \lesssim e^{-cL}\|P_Lu\|_{L^{p_B,1}_{r}L^2_\omega}
  \end{equation*}
  where the last step is the Lorentz duality pairing
  $L^{p_B',\infty}(d\nu)\times L^{p_B,1}(d\nu)\to\mathbb C$.
  Hence
  \begin{equation*}
    \|\one{C_j}\partial^\alpha R_1(z)P_Lu\|_{L^\infty_rL^2_\omega}
    \lesssim
    e^{-cL}\|P_Lu\|_{L^{p_B,1}_{r}L^2_\omega}.
  \end{equation*}
  Restricting the output to $C_j$ costs
  $(\int_{2^j}^{2^{j+1}}r^{n-1}\,dr)^{1/p_B'}$, which is absorbed
  by the exponential because $L\ge C(1+2^j)$. This proves
  \eqref{eq:highangtail}.

  Now take $N\simeq C(1+2^j)$ and decompose the input angular 
  variables by writing $I=\Pi_N+\sum_{L\ge N}P_L$. By
  \eqref{eq:FSendpoint}, \eqref{eq:angbern}, and
  \eqref{eq:highangtail},
  \begin{equation*}
    \|\one{C_j}\partial^\alpha R_1(z)\Pi_N\phi\|
      _{L^{p_B',\infty}_{r}L^2_\omega}
    \lesssim N^{\sigma_n}\|\phi\|_{L^{p_B,1}},
  \end{equation*}
  and
  \begin{equation*}
    \|\one{C_j}\partial^\alpha R_1(z)(I-\Pi_N)\phi\|
      _{L^{p_B',\infty}_{r}L^2_\omega}
    \lesssim
    \sum_{L\ge N}e^{-cL}
      \|P_L\phi\|_{L^{p_B,1}_{r}L^2_\omega}
    \lesssim
    \sum_{L\ge N}e^{-cL}L^{\sigma_n}
      \|\phi\|_{L^{p_B,1}}.
  \end{equation*}
  The last sum is bounded by
  $C N^{\sigma_n}\|\phi\|_{L^{p_B,1}}$. Since
  $N\simeq1+2^j$, this proves
  \eqref{eq:FSendpointloc1}.
\end{proof}

\subsection{Estimates for \pdfmath{R_{2}}}\label{sub:estR1}

The symbol of $R_{2}(z)$
\begin{equation*}
  R_{2}(\xi)=\frac{\psi^{c}(\xi)}{|\xi|^{2}-z}
\end{equation*} 
is smooth on $\mathbb{R}^{n}$ and behaves like
$\simeq \bra{\xi}^{-2}$. 
Thus $R_{2}(z)$ satisfies estimates equivalent
to Sobolev embedding with a loss of two derivatives. 
In the case $\alpha=0$, estimate \eqref{eq:estR2sob} below
is valid for indices $p,q$ in the closed region
$OKK'O'$ with the exclusion of the points $K$, $K'$
(this region becomes the triangle $OKO'$ in dimension $n=2$).
Recall the notation $\theta(p,q)=\frac n2(\frac 1p-\frac 1q)-1$.

\begin{proposition}\label{pro:estR2}
  Let $n\ge2$,
  $p,q\in(1,\infty)$, $r\in(0,\infty]$ and $|\alpha|\le1$.
  Then we have
  \begin{equation}\label{eq:estR2sob}
    \textstyle
    \|\partial^{\alpha}R_{2}(z)\phi\|_{L^{q,r}}
    \lesssim
    |z|^{\frac{|\alpha|}{2}+\theta(p,q)}
    \|\phi\|_{L^{p,r}}
    \qquad\text{for}\quad 
    0<\frac 1q \le\frac 1p\le \frac 1q+\frac{2-|\alpha|}{n}.
  \end{equation}
  Moreover we have, for $\nu\in[0,n]$,
  \begin{equation}\label{eq:estR2gut}
    \||x|^{-\nu}\partial^{\alpha}R_{2}(z)\phi\|_{\ell^{\infty}L^{q}}
    \lesssim
    |z|^{\frac{|\alpha|}{2}+\frac\nu2+\theta(p,q)}
    \|\phi\|_{L^{p}}
  \end{equation}
  provided
  \begin{equation}\label{eq:range1}
    \textstyle
    0<\frac 1q-\frac{\nu}{n}\le \frac 1p\le
    \frac 1q+\frac{2-\nu-|\alpha|}{n}.
  \end{equation}
  Finally we have
  \begin{equation}\label{eq:estR2gutb}
    \|\bra{x}^{-1/2}\partial^{\alpha}R_{2}(z)\phi\|
      _{\ell^{\infty}L^{q}}
    \lesssim
    |z|^{\frac{|\alpha|}{2}+\epsilon+\theta(p,q)}
    \|\phi\|_{L^{p}}
  \end{equation}
  provided
  \begin{equation*}
    \textstyle
    0<\frac 1q-\frac{1}{2n}\le \frac 1p\le
    \frac 1q+\frac{2-|\alpha|}{n}
  \end{equation*}
  and $\epsilon$ satisfies
  \begin{equation*}
    \textstyle
    \frac n2(\frac 1q-\frac 1p)_{+}\le \epsilon\le 
    \min\{\frac 14,-\theta(p,q)-\frac{|\alpha|}{2}\}.
  \end{equation*}
\end{proposition}

\begin{proof}
  We shall use the standard Sobolev inequalities
  in dimension $n\ge2$
  \begin{equation*}
    \|\bra{D}^{-k}u\|_{L^{s}}\lesssim\|u\|_{L^{p}},
    \qquad
    k=1,2
  \end{equation*}
  valid for all $p,s\in[1,\infty]$ with
  \begin{equation}\label{eq:sobin2}
    \textstyle
    0\le \frac 1s \le\frac 1p\le \frac 1s+\frac kn
    \quad\text{with}\quad 
    (\frac 1p,\frac 1s)\neq(\frac kn,0)
    \ \text{or}\  
    (1, 1-\frac kn).
  \end{equation}

  By scaling, it is sufficient to prove the claims for $|z|=1$.
  The symbol $m(\xi)=R_{2}(\xi)\bra{\xi}^{2}$ satisfies the 
  Mikhlin--H\"{o}rmander conditions uniformly in
  $|z|=1$, hence it is bounded on $L^{p}$ for all
  $p\in(1,\infty)$.
  Writing $R_{2}(z)=m(D)\bra{D}^{-2}$ we obtain
  \begin{equation*}
    \textstyle
    \|R_{2}(z)\phi\|_{L^{q}}\lesssim
    \|\bra{D}^{-2}\phi\|_{L^{q}}
    \lesssim
    \|\phi\|_{L^{p}}
  \end{equation*}
  for $p$ and $q=s$ as in \eqref{eq:sobin2}. By real interpolation,
  we obtain \eqref{eq:estR2sob} for $\alpha=0$.
  A similar argument
  works for $\partial R_{2}(z)=m_{1}(D)\bra{D}^{-1}$
  and $\partial^{2} R_{2}(z)=m_{2}(D)$ and gives 
  \eqref{eq:estR2sob} for $|\alpha|=1,2$.

  To prove \eqref{eq:estR2gut}, we use H\"{o}lder's 
  inequality with $\frac 1s=\frac 1q-\frac{\nu}{n}$,
  $\ell^{s}\hookrightarrow \ell^{\infty}$, and
  then Sobolev embedding as before:
  \begin{equation*}
    \||x|^{-\nu}\partial^{\alpha}R_{2}(z)\phi\|
      _{\ell^{\infty}L^{q}}\lesssim
    \||x|^{-\nu}\|_{\ell^{\infty}L^{n/\nu}}
    \|\partial^{\alpha}R_{2}(z)\phi\|
      _{\ell^{\infty}L^{s}}\lesssim
    \|\partial^{\alpha}R_{2}(z)\phi\| _{L^{s}}
    \lesssim \|\phi\|_{L^{p}}
  \end{equation*}
  provided the indices satisfy $p,q\in(1,\infty)$,
  $\frac 1s=\frac 1q-\frac{\nu}{n}$ and
  \begin{equation*}
    \textstyle
    0<\frac 1q-\frac{\nu}{n}\le \frac 1p\le 
    \frac 1q-\frac{\nu}{n}+\frac{2-|\alpha|}{n}
  \end{equation*}
  which gives the conditions stated in the claim.

  In \eqref{eq:estR2gutb} the weight is not homogeneous and
  we modify the previous argument as follows.
  We use the notation
  $\bra{x}_{t}=S_{1/t}(1+|x|^{2})^{1/2}
    =(1+\frac{|x|^{2}}{t^{2}})^{1/2}$ for $t>0$;
  note that
  \begin{equation*}
    \|\bra{x}_{t}^{-1/2}\|_{\ell^{\infty}L^{r}}=
    ct^{\frac nr},
    \qquad
    c=\|\bra{x}^{-1/2}\|_{\ell^{\infty}L^{r}}<\infty
    \quad\text{for}\quad r\in[2n,\infty].
  \end{equation*}
  For $|z|=1$ we may write
  \begin{equation*}
    \|\bra{x}_{t}^{-1/2}\partial^{\alpha}R_{2}(z)\phi\|
      _{\ell^{\infty}L^{q}}\le 
    \|\bra{x}_{t}^{-1/2}\|_{\ell^{\infty}L^{r}}
    \|\partial^{\alpha}R_{2}(z)\phi\|_{\ell^{\infty}L^{s}}\lesssim
    t^{\frac nr} \|\phi\|_{L^{p}}
  \end{equation*}
  with $r\in[2n,\infty]$, $q\in[1,r]$,
  $\frac 1s=\frac 1q-\frac 1r$
  and $s,p$ satisfying \eqref{eq:sobin2} with $k=2-|\alpha|$.
  Now, recalling \eqref{eq:scaR0}, we may write
  \begin{equation*}
    \bra{x}^{-1/2}\partial^{\alpha}R_{2}(z)=
    t^{|\alpha|-2}S_{t}\bra{x}_{t}^{-1/2}
    \partial^{\alpha}R_{2}(t^{-2}z)S_{1/t}
  \end{equation*}
  and choosing $t=|z|^{1/2}$, from the last
  inequality we obtain for arbitrary $z\in \rho(-\Delta)$
  \begin{equation*}
    \|\bra{x}^{-1/2}\partial^{\alpha}R_{2}(z)\phi\|
      _{\ell^{\infty}L^{q}}\lesssim
    |z|^{\frac{|\alpha|}{2}+\frac{n}{2r}+\theta(p,q)}
    \|\phi\|_{L^{p}}
  \end{equation*}
  with $r\in[2n,\infty]$, $q\in[1,r)$, $p\in[1,\infty)$ and
  \begin{equation*}
    \textstyle
    0<\frac 1q-\frac 1r \le\frac 1p\le 
    \frac 1q-\frac 1r+\frac{2-|\alpha|}n.
  \end{equation*}
  Writing $\epsilon=\frac{n}{2r}\in[0,\frac 14]$ we see that
  the last estimate is valid provided
  \begin{equation*}
    \textstyle
    0<\frac 1q-\frac{1}{2n}\le \frac 1p\le
    \frac 1q+\frac{2-|\alpha|}{n}
  \end{equation*}
  with $\epsilon$ satisfying
  \begin{equation*}
    \textstyle
    \frac n2(\frac 1q-\frac 1p)_{+}\le \epsilon\le 
    \min\{\frac 14,-\theta(p,q)-\frac{|\alpha|}{2}\}
  \end{equation*}
  and this proves the last claim.
\end{proof}

\begin{corollary}[]\label{cor:betterlor}
  In the open rectangle $OKK'O'$ estimate \eqref{eq:estR2sob}
  can be improved to
  \begin{equation}\label{eq:estR2sobimpr}
    \textstyle
    \|\partial^{\alpha}R_{2}(z)\phi\|_{L^{q,1}}
    \lesssim
    |z|^{\frac{|\alpha|}{2}+\theta(p,q)}
    \|\phi\|_{L^{p,\infty}}
    \qquad\text{for}\quad 
    1> \frac 1p> \frac 1q> \frac 1p-\frac{2-|\alpha|}{n}>0
  \end{equation}
\end{corollary}

\begin{proof}
  We can assume $|z|=1$.
  We write \eqref{eq:estR2sob} at two different $q$'s
  for the same $p$:
  \begin{equation*}
    \|\partial^{\alpha}R_{2}(z)\phi\|_{L^{q\pm\epsilon,r}}
    \lesssim
    \|\phi\|_{L^{p,r}}
  \end{equation*}
  which is possible in view of the open condition
  \eqref{eq:estR2sobimpr} on the indices.
  Then we interpolate using
  $(L^{q-\epsilon,r},L^{q+\epsilon,r})_{\frac 12,1}= L^{q,1}$;
  this gives
  \begin{equation*}
    \|\partial^{\alpha}R_{2}(z)\phi\|_{L^{q,1}}
    \lesssim
    \|\phi\|_{L^{p,r}}
  \end{equation*}
  for all $p,q$ as in \eqref{eq:estR2sobimpr}.
  We write the last inequality for two different $p$'s 
  and the same $q$:
  \begin{equation*}
    \|\partial^{\alpha}R_{2}(z)\phi\|_{L^{q,1}}
    \lesssim
    \|\phi\|_{L^{p\pm \epsilon,r}}
  \end{equation*}
  and then we interpolate using 
  $(L^{p-\epsilon,r},L^{p+\epsilon,r})_{\frac 12,\infty}= 
    L^{p,\infty}$,
  and we get \eqref{eq:estR2sobimpr}.
\end{proof}

\subsection{Bounds for the full resolvent \pdfmath{R_0}}
\label{sub:boun_full_reso}

Here we compare the estimates of the previous 
Sections with the known results for $R_{0}(z)$.
We recall that $n\ge2$ and
\begin{equation*}
  \textstyle
  \frac{1}{p_{B}}=\frac 12+\frac{1}{2n},
  \quad
  \frac{1}{p_{B}'}=\frac 12-\frac{1}{2n},
  \quad
  \frac{1}{q_{B}}=\frac 12+\frac{1}{2n}-\frac{2}{n+1},
  \quad
  \frac{1}{q_{B}'}=\frac 12-\frac{1}{2n}+\frac{2}{n+1}.
\end{equation*}
As usual, we refer to the figure in the Introduction for 
a graphical representation of the conditions on $(p,q)$.
In \cite{KenigRuizSogge87-a}, \cite{Gutierrez04-a} it is proved
that
\begin{equation}\label{eq:KRS}
  \textstyle
  \|R_{0}(z)\phi\|_{L^{q}}\lesssim|z|^{\theta(p,q)}\|\phi\|_{L^{p}},
  \qquad
  \theta(p,q)=\frac n2(\frac 1p-\frac 1q)-1
\end{equation}
for $(p,q)$
in the open quadrilateral $ABB'A'$ plus the sides $AA'$, $BB'$:
\begin{equation*}
  \textstyle
  \frac{2}{n+1}\le\frac 1p-\frac 1q\le\frac 2n,
  \quad
  \frac{1}{p_{B}}<\frac 1p,
  \quad
  \frac 1q<\frac{1}{p_{B}'}.
\end{equation*}
Although not clearly stated in the original references,
the techniques of these papers work for all dimensions $n\ge2$,
as remarked in \cite{KwonLee20-a}.
In \cite{Gutierrez04-a} it is also proved that
on the sides $AB$, $A'B'$ the strong estimate 
\eqref{eq:KRS} can be replaced by a weak 
$L^{p,1}\to L^{q,\infty}$ estimate.
Combining Theorem \ref{the:estR1} and estimate
\eqref{eq:estR2sob},
we obtain the following precised estimates
(in each case, the power of $|z|$ is $\theta(p,q)$):
\begin{itemize}
  \item points $B$, $B'$:
  \begin{equation}\label{eq:R0B}
    \|R_{0}(z)\phi\|_{L^{q_{B},\infty}}\lesssim
    |z|^{-\frac{1}{n+1}}\|\phi\|_{L^{p_{B},1}},
    \qquad
    \|R_{0}(z)\phi\|_{L^{p_{B}',\infty}}\lesssim
    |z|^{-\frac{1}{n+1}}\|\phi\|_{L^{q_{B}',1}}
  \end{equation}
  \item line $BA$:
  \begin{equation}\label{eq:R0BA}
    \textstyle
    \|R_{0}(z)\phi\|_{L^{q,1}}\lesssim
    |z|^{\frac n2(\frac 12-\frac 1q)-\frac 34}
    \|\phi\|_{L^{p_{B},1}},
    \qquad
    \frac 12-\frac{3}{2n}\le\frac 1q<\frac{1}{q_{B}}
  \end{equation}
  \item line $B'A'$:
  \begin{equation}\label{eq:R0BpAp}
    \textstyle
    \|R_{0}(z)\phi\|_{L^{p_{B}',\infty}}\lesssim
    |z|^{\frac n2(\frac 1{q'}-\frac 12)-\frac 34}
    \|\phi\|_{L^{q',\infty}},
    \qquad
    \frac{1}{q_{B}'}<\frac{1}{q'}\le \frac 12+\frac{3}{2n}
  \end{equation}
  \item line $BB'$: for all $r\in(0,\infty]$,
  \begin{equation}\label{eq:R0BBp}
    \textstyle
    \|R_{0}(z)\phi\|_{L^{q,r}}\lesssim
    |z|^{-\frac{1}{n+1}}\|\phi\|_{L^{p,r}},
    \qquad
    \frac 1p-\frac 1q=\frac 2{n+1},
    \quad
    \frac{1}{p_{B}}<\frac 1p<\frac{1}{q_{B}'}
  \end{equation}
  \item line $AA'$: for all $r\in(0,\infty]$,
  \begin{equation}\label{eq:R0AAp}
    \textstyle
    \|R_{0}(z)\phi\|_{L^{q,r}}\lesssim\|\phi\|_{L^{p,r}},
    \qquad
    \frac 1p-\frac 1q=\frac 2{n},
    \quad
    \frac{1}{p_{B}}<\frac 1p<\frac 12+\frac{3}{2n}
  \end{equation}
  \item open region $ABB'A'$:
  \begin{equation}\label{eq:ABBpAp}
    \textstyle
    \|R_{0}(z)\phi\|_{L^{q,1}}\lesssim
    |z|^{\theta(p,q)}\|\phi\|_{L^{p,\infty}},
    \qquad
    \frac{2}{n+1}<\frac 1p-\frac 1q<\frac 2n,
    \quad
    \frac{1}{p_{B}}<\frac 1p,
    \quad
    \frac 1q<\frac{1}{p_{B}'}.
  \end{equation}
\end{itemize}
The estimates are truely independent of $z$ only when 
$\theta(p,q)=0$ i.e.~on the line $AA'$.

Consider now weighted estimates. Using dyadic norms,
Theorems 7 and 8 of \cite{Gutierrez04-a}
(see also \cite{RuizVega93-a}) 
can be reformulated in an equivalent way as follows:
for $n\ge3$, $|\alpha|\le1$, $p\in(1,\infty)$
and frequency $|z|=1$
\begin{equation}\label{eq:th78gut}
  \textstyle
  \|\bra{x}^{-\frac 12}\partial^{\alpha}R_{0}(z)\phi\|
    _{\ell^{\infty}L^{2}}
  \lesssim
  \|\phi\|_{L^{p}},
  \qquad
  \frac 12+\frac{1}{n+1}\le \frac 1p\le 
  \frac 12+\frac{2-|\alpha|}{n}.
\end{equation}
After rescaling, one gets an estimate for all
$z$ with a power $|z|^{\frac{|\alpha|}{2}+\frac 14+\theta(p,q)}$,
but since the weight is not homogeneous the resulting norm
depends on the frequency (see formulas (19)--(20) in
\cite{Gutierrez04-a}).

If we combine \eqref{eq:estR1gut} and \eqref{eq:estR2gut}
with the choice $\nu=1/2$, we obtain the following
estimate with a homogeneous weight: for $n\ge2$,
$p\in(1,\infty)$, $q\in[2,2n]$ and all $z$
\begin{equation}\label{eq:weiR0}
  \textstyle
  \||x|^{-\frac 12}R_{0}(z)\phi\| _{\ell^{\infty}L^{q}}
  \lesssim
  |z|^{\frac 14+\theta(p,q)}
  \|\phi\|_{L^{p}},
  \qquad
  \frac 12+\frac{1}{n+1}\le \frac 1p\le \frac 1q+\frac{3}{2n}.
\end{equation}
We do not get any estimates for $\partial^{\alpha}R_{0}$
for $|\alpha|\ge1$.
Note that the range of indices is restricted with respect to
\eqref{eq:th78gut}.

If we use instead an inhomogeneous weight $\bra{x}^{-\frac 12}$,
we can combine \eqref{eq:estR1gut} and
\eqref{eq:estR2gutb}, 
and for the range
\begin{equation*}
  \textstyle
  n\ge2,
  \qquad
  p\in(1,\infty),
  \qquad
  q\in[2,2n],
  \qquad
  \frac 12+\frac{1}{n+1}\le \frac 1p\le \frac 1q+\frac{2-|\alpha|}{n},
\end{equation*}
we improve \cite{Gutierrez04-a} as follows
(note that $\theta(p,q)=
  \frac n2(\frac 1p-\frac 1q)-1\le-\frac{|\alpha|}{2}$):
\begin{equation}\label{eq:bettgut}
  \textstyle
  \|\bra{x}^{-\frac 12}\partial^{\alpha}R_{0}(z)\phi\|
    _{\ell^{\infty}L^{q}}
  \lesssim
  |z|^{\frac{|\alpha|}{2}+\epsilon+\theta(p,q)}\|\phi\|_{L^{p}}
  \quad\text{for all}\quad 
  0\le \epsilon\le-\frac{|\alpha|}{2}-\theta(p,q).
\end{equation}

\section{Estimates for the perturbed resolvent}

We shall need the following estimate
for the perturbed resolvent $R(z)$.  Its strip version was proved
in \cite{DAncona20}; selfadjointness then extends it to the full
resolvent set, as explained below.
Several variants of this estimate exist in the literature,
(see e.g. \cite{BarceloVegaZubeldia13-a}, \cite{Zubeldia14-a},
\cite{CardosoCuevasVodev13-a});
this version has the advantage of being uniform 
and scale invariant in $z\in \mathbb{C}$.

\begin{lemma}[]\label{lem:resolvestR}
  Under Assumption \textbf{(H)} we have 
  \begin{equation}\label{eq:resolvestR}
    \||x|^{-\frac32}R(z)\phi\|_{\ell^{\infty}L^{2}}+
    |z|^{1/2} \||x|^{-\frac 12}R(z)\phi\|_{\ell^{\infty}L^{2}}+
    \||x|^{-\frac 12}\partial R(z)\phi\|_{\ell^{\infty}L^{2}}\le C
    \||x|^{\frac 12}\phi\|_{\ell^{1}L^{2}}
  \end{equation}
  with a constant $C$ uniform in
  $z\in\mathbb{C}\setminus[0,+\infty)$.
\end{lemma}

\begin{proof}
  In \cite{DAncona20} estimate \eqref{eq:resolvestR} is proved
  under the additional restriction $|\Im z|\le1$,
  which can be easily removed using the selfadjointness of $H$;
  we explain how to do this. We restrict to the upper
  half plane, since the argument for the lower half plane is 
  identical.  Let
  $z=x+iy$ with $y>1$.
  For $f,g\in C_{c}^{\infty}(\mathbb{R}^{n})$, set
  \begin{equation*}
    F(\zeta)=\langle R(\zeta)f,g\rangle.
  \end{equation*}
  Then $G(w)=F(w+i)$ is analytic and bounded in $\Im w>0$.
  Indeed, by the spectral theorem,
  \begin{equation*}
    |G(w)|\le
    \|R(w+i)\|_{L^{2}\to L^{2}}\|f\|_{L^{2}}\|g\|_{L^{2}}
    \le
    \|f\|_{L^{2}}\|g\|_{L^{2}}.
  \end{equation*}
  Hence $G$ is represented by the Poisson integral of its boundary
  values.  With
  $P_{a}(s)=\pi^{-1}a(s^{2}+a^{2})^{-1}$, $a>0$, we get
  \begin{equation*}
    |F(x+iy)|
    \le
    \int_{\mathbb{R}}P_{y-1}(x-t)|F(t+i)|\,dt .
  \end{equation*}
  Applying the strip estimate on the line $\Im z=1$ and dualizing
  the dyadic norms gives, respectively,
  \begin{equation*}
    |\langle R(t+i)f,g\rangle|
    \lesssim
    \||x|^{\frac12}f\|_{\ell^{1}L^{2}}
    \||x|^{\frac32}g\|_{\ell^{1}L^{2}},
  \end{equation*}
  \begin{equation*}
    |\langle R(t+i)f,g\rangle|
    \lesssim
    |t+i|^{-\frac12}
    \||x|^{\frac12}f\|_{\ell^{1}L^{2}}
    \||x|^{\frac12}g\|_{\ell^{1}L^{2}}.
  \end{equation*}
  The elementary estimate
  \begin{equation*}
    \int_{\mathbb{R}}P_{y-1}(x-t)|t+i|^{-\frac12}\,dt
    \lesssim |x+iy|^{-\frac12},
    \qquad y>1,
  \end{equation*}
  follows by splitting the integral into
  $|t-x|\le y$ and $|t-x|>y$ after the elementary scaling
  $t=x+ys$.  Thus the first two terms in \eqref{eq:resolvestR} follow
  for $\Im z>1$.

  For the derivative term we test against
  $g=(g_{1},\ldots,g_{n})\in C_{c}^{\infty}(\mathbb{R}^{n};\mathbb{C}^{n})$
  and write
  \begin{equation*}
    \langle \partial R(\zeta)f,g\rangle
    =
    -\langle R(\zeta)f,\partial \cdot g\rangle .
  \end{equation*}
  Thus $\langle \partial R(\zeta)f,g\rangle$ is again a scalar
  resolvent matrix element.  The spectral theorem gives the
  required boundedness in the shifted half-plane, while the strip
  estimate gives
  \begin{equation*}
    |\langle \partial R(t+i)f,g\rangle|
    \lesssim
    \||x|^{\frac12}f\|_{\ell^{1}L^{2}}
    \||x|^{\frac12}g\|_{\ell^{1}L^{2}}.
  \end{equation*}
  The Poisson formula therefore gives the same bound at
  $z=x+iy$, $y>1$. The desired norm estimates follow
  by density and duality for the dyadic norms.
\end{proof}

\subsection{Proof of Theorem \ref{the:2}} \label{sub:the_2}

First of all, we notice that estimate \eqref{eq:UREH} follows
from \eqref{eq:UREweak} by duality and interpolation.
Thus it is sufficient to consider indices of the form
\begin{equation*}
  \textstyle
  (\frac 1p,\frac 1q)=(\frac{1}{p_{B}},\frac 1q)\in AB.
\end{equation*}
For convenience, we rewrite the operator $H$ as
\begin{equation}\label{eq:rewrH}
  H=-\Delta+ a\cdot \partial+b,
  \qquad
  a=2iA,\qquad
  b=i(\partial \cdot A)+|A|^{2}+V.
\end{equation}
Then the resolvent $R(z)$  can be decomposed as follows:
\begin{equation}\label{eq:decomp}
  R=R_{0}-R_{0}(a \cdot \partial+b)R_{0}
  +R_{0}(a \cdot \partial+b)R(a \cdot \partial+b)R_{0}.
\end{equation}
By expanding $H+\Delta=a \cdot \partial+b$ we obtain several
terms
\begin{equation}\label{eq:decI}
  R=R_{0}-I_{1}-I_{2}+
  II_{1}+II_{2}+II_{3}+II_{4}
\end{equation}
where
\begin{equation*}
  I_{1}=R_{0}a \cdot \partial R_{0},
  \qquad
  I_{2}=R_{0}bR_{0}
\end{equation*}
\begin{equation*}
  II_{1}=R_{0}a \cdot \partial R a \cdot \partial R_{0},
  \qquad
  II_{2}=R_{0}b R b R_{0}
\end{equation*}
\begin{equation*}
  II_{3}=R_{0}a \cdot \partial R b R_{0},
  \qquad
  II_{4}=R_{0}bRa \cdot \partial R_{0}.
\end{equation*}

To estimate the first term $I_{1}$, we split it as follows:
\begin{equation*}
  I_{1}=
  R_{0}a \cdot \partial R_{0}=
  R_{0}a \cdot \partial R_{1}+
  R_{0}a \cdot \partial R_{2}.
\end{equation*}
By the dual of estimate \eqref{eq:weiR0} we have
\begin{equation*}
  \|R_{0}a \cdot \partial R_{1}\phi\|_{L^{q,\infty}}
  \lesssim
  |z|^{\frac 14+\theta(q',2)}
  \||x|^{1/2}a \partial R_{1}\phi\|_{\ell^{1}L^{2}}
\end{equation*}
for all $q\in[q_{B},q_{A}]$ and actually in the larger
range $q\in[q_{B},q_{D}]$.
Then by H\"{o}lder
\begin{equation*}
  \textstyle
  \lesssim
  |z|^{\frac 14+\theta(q',2)}
  \||x|^{1/2}a\|_{\ell^{q_{B}'}L^{r,2}}
  \|\partial R_{1}\phi\|_{L^{q_{B},\infty}},
  \qquad
  \frac 1r=\frac{2}{n+1}-\frac{1}{2n}
\end{equation*}
and using \eqref{eq:estR1aweak} we obtain
\begin{equation*}
  \textstyle
  \|R_{0}a \cdot\partial R_{1}\phi\|_{L^{q,\infty}}
  \lesssim
  |z|^{\mu_{1}}
  \||x|^{1/2}a\|_{\ell^{q_{B}'}L^{r,2}}
  \|\phi\|_{L^{p_{B},1}},
  \qquad
  \mu_{1}=\frac n2(\frac 12-\frac 1q)-\frac 14-\frac{1}{n+1}.
\end{equation*}

Under the alternative assumption \eqref{eq:assA2FS} we estimate
the same term by using the endpoint Frank--Simon lemma.  Put
$\sigma_{n}=(n-1)/(2n)$ and $u=\partial R_{1}(z)\phi$.
We first record the dyadic estimate
\begin{equation}\label{eq:FSaR1}
  \||x|^{1/2}a\,u\|_{\ell^{1}L^{2}}
  \lesssim
  (1+|z|^{\sigma_{n}/2})\|\phi\|_{L^{p_{B},1}}.
\end{equation}
By the localized estimate \eqref{eq:FSendpointloc} in
Lemma \ref{lem:FSendpointloc},
\begin{equation*}
  \|\one{C_{j}}u\|_{L^{p_{B}',\infty}_{r}L^{2}_{\omega}}
  \lesssim
  (1+|z|^{1/2}2^{j})^{\sigma_{n}}\|\phi\|_{L^{p_{B},1}}.
\end{equation*}
Since $p_{B}'=2n/(n-1)$, on the radial interval corresponding to
$C_{j}$ we have
\begin{equation*}
  \|\one{C_{j}}u\|_{L^{2}}
  \lesssim
  2^{j/2}
  \|\one{C_{j}}u\|_{L^{p_{B}',\infty}_{r}L^{2}_{\omega}}.
\end{equation*}
Moreover \eqref{eq:assA2FS} gives, for $j\ge0$,
\begin{equation*}
  \||x|^{1/2}a\|_{L^{\infty}(C_{j})}
  \lesssim 2^{-j(1/2+\sigma_{n}+\varepsilon)},
\end{equation*}
while \textbf{(H)} gives, for $j<0$,
\begin{equation*}
  \||x|^{1/2}a\|_{L^{\infty}(C_{j})}
  \lesssim 2^{-j/2}\bra{j}^{-\mu}.
\end{equation*}
Combining these estimates,
\begin{equation*}
  \sum_{j\ge0}
  \|\one{C_{j}}|x|^{1/2}a u\|_{L^{2}}
  \lesssim
  \sum_{j\ge0}2^{-j(\sigma_{n}+\varepsilon)}
  (1+|z|^{1/2}2^{j})^{\sigma_{n}}
  \|\phi\|_{L^{p_{B},1}}
  \lesssim
  (1+|z|^{\sigma_{n}/2})\|\phi\|_{L^{p_{B},1}},
\end{equation*}
where we used $(1+t)^{\sigma_{n}}\lesssim1+t^{\sigma_{n}}$.
Similarly, since $\mu>1$,
\begin{equation*}
  \sum_{j<0}
  \|\one{C_{j}}|x|^{1/2}a u\|_{L^{2}}
  \lesssim
  \sum_{j<0}\bra{j}^{-\mu}(1+|z|^{1/2}2^{j})^{\sigma_{n}}
  \|\phi\|_{L^{p_{B},1}}
  \lesssim
  (1+|z|^{\sigma_{n}/2})\|\phi\|_{L^{p_{B},1}}.
\end{equation*}
This proves \eqref{eq:FSaR1}.
Hence the first piece of $I_{1}$ satisfies, under
\eqref{eq:assA2FS},
\begin{equation*}
  \|R_{0}a \cdot\partial R_{1}\phi\|_{L^{q,\infty}}
  \lesssim
  |z|^{\mu_{2}}(1+|z|^{\sigma_{n}/2})
  \|\phi\|_{L^{p_{B},1}},
  \qquad
  \mu_{2}=\frac n2(\frac 12-\frac 1q)-\frac 34.
\end{equation*}
For the second piece of $I_{1}$,
by \eqref{eq:R0BA} and H\"{o}lder we get
\begin{equation*}
  \|R_{0}a \cdot \partial R_{2}\phi\|_{L^{q,\infty}}
  \lesssim
  |z|^{\mu_{2}}
  \|a \cdot \partial R_{2}\phi\|_{L^{p_{B},1}}
  \lesssim
  |z|^{\mu_{2}}
  \||x|^{1/2}a\|_{\ell^{p_{B}}L^{2n,2}}
  \||x|^{-1/2}\partial R_{2}\phi\|_{\ell^{\infty}L^{2}}
\end{equation*}
with $\mu_{2}=\frac n2(\frac 12-\frac 1q)-\frac 34$,
and by \eqref{eq:estR2gut} we get
\begin{equation*}
  \textstyle
  \|R_{0}a \cdot \partial R_{2}\phi\|_{L^{q,\infty}}
  \lesssim
  |z|^{\mu_{2}}
  \||x|^{1/2}a\|_{\ell^{p_{B}}L^{2n,2}}
  \|\phi\|_{L^{p_{B}}},
  \qquad
  \mu_{2}=\frac n2(\frac 12-\frac 1q)-\frac 34=\theta(p_{B},q).
\end{equation*}
Under \eqref{eq:assA2}, summing up we get for $I_{1}$
\begin{equation}\label{eq:estI1}
  \textstyle
  \|I_{1}\phi\|_{L^{q,\infty}}
  \lesssim
  (|z|^{\mu_{1}}+|z|^{\mu_{2}})
  \left[
    \||x|^{1/2}a\|_{\ell^{q_{B}'}L^{r,2}}+
    \||x|^{1/2}a\|_{\ell^{p_{B}}L^{2n,2}}
  \right]
  \|\phi\|_{L^{p_{B},1}},
  \quad
  \frac 1r=\frac{2}{n+1}-\frac{1}{2n}.
\end{equation}
Term $I_{2}$: by \eqref{eq:R0BA} and H\"{o}lder we get
(recall $\mu_{2}=\theta(p_{B},q)$)
\begin{equation}\label{eq:estI2}
  \|I_{2}\phi\|_{L^{q,\infty}}\lesssim
  |z|^{\mu_{2}}
  \|bR_{0}\phi\|_{L^{p_{B},1}}\lesssim
  |z|^{\mu_{2}}
  \|b\|_{L^{\frac n2,1}}\|R_{0}\phi\|_{L^{q_{A},\infty}}\lesssim
  |z|^{\mu_{2}}
  \|b\|_{L^{\frac n2,1}}\|\phi\|_{L^{p_{B},1}}.
\end{equation}
Consider now the term $II_{1}$, at first for $|z|=1$.
Splitting $R_{0}=R_{1}+R_{2}$, we estimate as for the term $I_{1}$
\begin{equation*}
  \|R_{0}a \cdot \partial R a\cdot \partial R_{1}\phi\|_{L^{q,\infty}}
  \lesssim
  \|a \cdot \partial R a\cdot \partial R_{1}\phi\|_{L^{p_{B},1}}
  \lesssim
  \||x|^{1/2}a\|_{\ell^{p_{B}}L^{2n,2}}
  \||x|^{-1/2}\partial R a\cdot \partial R_{1}\phi\|
     _{\ell^{\infty}L^{2}}
\end{equation*}
and by the resolvent estimate \eqref{eq:resolvestR}
\begin{equation*}
  \lesssim
  \||x|^{1/2}a\|_{\ell^{p_{B}}L^{2n,2}}
  \||x|^{1/2}a \cdot \partial R_{1}\phi\|_{\ell^{1}L^{2}}.
\end{equation*}
We estimate the last term exactly as above:
\begin{equation*}
  \textstyle
  \||x|^{1/2}a \cdot \partial R_{1}\phi\|_{\ell^{1}L^{2}}
  \lesssim
  \||x|^{1/2}a\|_{\ell^{q_{B}'}L^{r,2}}
  \|\partial R_{1}\phi\|_{L^{q_{B},\infty}},
  \qquad
  \frac 1r=\frac{2}{n+1}-\frac{1}{2n}
\end{equation*}
and finally using \eqref{eq:estR1aweak}
\begin{equation*}
  \|R_{0}a \cdot \partial R a\cdot \partial R_{1}\phi\|_{L^{q,\infty}}
  \lesssim
  \||x|^{1/2}a\|_{\ell^{p_{B}}L^{2n,2}}
  \||x|^{1/2}a\|_{\ell^{q_{B}'}L^{r,2}}
  \|\phi\|_{L^{p_{B},1}}.
\end{equation*}
For the second piece of $II_{1}$ we have as above
\begin{equation*}
  \|R_{0}a \cdot \partial R a\cdot \partial R_{2}\phi\|_{L^{q,\infty}}
  \lesssim
  \||x|^{1/2}a\|_{\ell^{p_{B}}L^{2n,2}}
  \||x|^{1/2}a \cdot \partial R_{2}\phi\|_{\ell^{1}L^{2}}
\end{equation*}
and by H\"{o}lder
\begin{equation*}
  \lesssim
  \||x|^{1/2}a\|_{\ell^{p_{B}}L^{2n,2}}
  \||x|^{1/2}a\|_{\ell^{p_{B}}L^{2n}}
  \|\partial R_{2}\phi\|_{L^{\frac{2n}{n-1}}}.
\end{equation*}
Using \eqref{eq:estR2sob} we conclude
\begin{equation}\label{eq:estII12}
  \lesssim
  \||x|^{1/2}a\|_{\ell^{p_{B}}L^{2n,2}}
  \||x|^{1/2}a\|_{\ell^{p_{B}}L^{2n}}
  \|\phi\|_{L^{p_{B}}}.
\end{equation}
Summing up, writing $\frac 1r=\frac{2}{n+1}-\frac{1}{2n}$,
term $II_{1}$ can be estimated as
\begin{equation}\label{eq:estII1}
  \|II_{1}\phi\|_{L^{q,\infty}}\lesssim
    \||x|^{1/2}a\|_{\ell^{p_{B}}L^{2n,2}}
  \left[
    \||x|^{1/2}a\|_{\ell^{q_{B}'}L^{r,2}}
    +
    \||x|^{1/2}a\|_{\ell^{p_{B}}L^{2n}}
  \right]
  \|\phi\|_{L^{p_{B},1}}.
\end{equation}
If we consider instead arbitrary frequencies $z$, keeping track
of the powers as for the term $I_{1}$, we obtain an additional
factor $|z|^{\mu_{1}}+|z|^{\mu_{2}}$ as in \eqref{eq:estI1}.

We estimate $II_{2}$ for $|z|=1$:
using \eqref{eq:R0BA} and H\"{o}lder
\begin{equation*}
  \|R_{0}b RbR_{0}\phi\|_{L^{q,\infty}}
  \lesssim
  \|b RbR_{0}\phi\|_{L^{p_{B},1}}
  \lesssim
  \||x|^{3/2}b\|_{\ell^{p_{B}}L^{2n,2}}
  \||x|^{-3/2}RbR_{0}\phi\|_{\ell^{\infty}L^{2}}
\end{equation*}
and using \eqref{eq:resolvestR}
\begin{equation*}
  \lesssim
  \||x|^{3/2}b\|_{\ell^{p_{B}}L^{2n,2}}
  \||x|^{1/2}bR_{0}\phi\|_{\ell^{1}L^{2}}
  \lesssim
  \||x|^{3/2}b\|_{\ell^{p_{B}}L^{2n,2}}
  \||x|^{1/2}b\|_{\ell^{q_{A}'}L^{2n/3,2}}
  \|R_{0}\phi\|_{L^{q_{A},\infty}}
\end{equation*}
and by \eqref{eq:R0BA} we conclude
\begin{equation}\label{eq:estII2}
  \|II_{2}\phi\|_{L^{q,\infty}}
  \lesssim
  \||x|^{3/2}b\|_{\ell^{p_{B}}L^{2n,2}}
  \||x|^{1/2}b\|_{\ell^{q_{A}'}L^{2n/3,2}}
  \|\phi\|_{L^{p_{B},1}}.
\end{equation}
If we consider $|z|\neq1$, keeping track of the powers of $|z|$
we get a factor $|z|^{\mu_{2}}=|z|^{\theta(p_{B},q)}$.

For $II_{3}$, when $|z|=1$, we may write as for $II_{1}$
\begin{equation*}
  \|R_{0}a \cdot \partial R b R_{0}\|_{L^{q,\infty}}
  \lesssim
  \||x|^{1/2}a\|_{\ell^{p_{B}}L^{2n,2}}
  \||x|^{1/2} b R_{0}\phi\| _{\ell^{1}L^{2}}
\end{equation*}
and then computing as for $II_{2}$
\begin{equation*}
  \lesssim
  \||x|^{1/2}a\|_{\ell^{p_{B}}L^{2n,2}}
  \||x|^{1/2}b\|_{\ell^{q_{A}'}L^{2n/3,2}}
  \|\phi\|_{L^{p_{B},1}}
\end{equation*}
and summing up
\begin{equation}\label{eq:estII3}
  \|II_{3}\phi\|_{L^{q,\infty}}\lesssim
  \||x|^{1/2}a\|_{\ell^{p_{B}}L^{2n,2}}
  \||x|^{1/2}b\|_{\ell^{q_{A}'}L^{2n/3,2}}
  \|\phi\|_{L^{p_{B},1}}.
\end{equation}
For $|z|\neq1$ we obtain a factor $|z|^{\mu_{2}}$.

Finally, for $II_{4}$ with $|z|=1$ we write like for $II_{2}$
\begin{equation*}
  \|R_{0}bRa \cdot \partial R_{0}\|_{L^{q,\infty}}\lesssim
  \||x|^{3/2}b\|_{\ell^{p_{B}}L^{2n,2}}
  \||x|^{-3/2}Ra\cdot\partial R_{0}\phi\|_{\ell^{\infty}L^{2}}
\end{equation*}
and by \eqref{eq:resolvestR}
\begin{equation*}
  \lesssim
  \||x|^{3/2}b\|_{\ell^{p_{B}}L^{2n,2}}
  \||x|^{1/2}a \cdot \partial R_{0}\phi\|_{\ell^{1}L^{2}}
\end{equation*}
and proceeding as for the term $II_{1}$ we get
\begin{equation*}
  \lesssim
  \||x|^{3/2}b\|_{\ell^{p_{B}}L^{2n,2}}
  \left[
    \||x|^{1/2}a\|_{\ell^{q_{B}'}L^{r,2}}
    +
    \||x|^{1/2}a\|_{\ell^{p_{B}}L^{2n}}
  \right]
  \|\phi\|_{L^{p_{B},1}}.
\end{equation*}
Summing up we get
\begin{equation}\label{eq:estII4}
  \|II_{4}\phi\|_{L^{q,\infty}}\lesssim
  \||x|^{3/2}b\|_{\ell^{p_{B}}L^{2n,2}}
  \left[
    \||x|^{1/2}a\|_{\ell^{q_{B}'}L^{r,2}}
    +
    \||x|^{1/2}a\|_{\ell^{p_{B}}L^{2n}}
  \right]
  \|\phi\|_{L^{p_{B},1}}.
\end{equation}
When $|z|\neq1$ we must include an additional factor
$|z|^{\mu_{1}}+|z|^{\mu_{2}}$.

Collecting the previous estimates we obtain
\begin{equation}\label{eq:finalest}
  \|R(z)\phi\|_{L^{q,\infty}}\lesssim
  C(a,b)(|z|^{\mu_{1}}+|z|^{\mu_{2}})\|\phi\|_{L^{p_{B},1}}
\end{equation}
where
\begin{equation*}
  C(a,b)=\alpha_{1}+\alpha_{2}+\beta_{1}+
  (\alpha_{1}+\alpha_{2}+\beta_{2})(\alpha_{2}+\beta_{3})
\end{equation*}
\begin{equation}\label{eq:can}
  \textstyle
  \alpha_{1}=\||x|^{1/2}a\|_{\ell^{q_{B}'}L^{r,2}},
  \quad
  \alpha_{2}=\||x|^{1/2}a\|_{\ell^{p_{B}}L^{2n,2}},
  \qquad
  \frac 1r=\frac{2}{n+1}-\frac{1}{2n},
\end{equation}
\begin{equation}\label{eq:cbn}
  \beta_{1}=\|b\|_{L^{n/2,1}},
  \quad
  \beta_{2}=
  \||x|^{1/2}b\|_{\ell^{q_{A}'}L^{2n/3,2}},
  \quad
  \beta_{3}=
  \||x|^{3/2}b\|_{\ell^{p_{B}}L^{2n,2}}.
\end{equation}
Note that $\alpha_{2},\beta_{1},\beta_{2}<\infty$
thanks to assumption \textbf{(H)},
while $\alpha_{1}<\infty$ by \textbf{(H)} and \eqref{eq:assA2},
which proves the estimate under \eqref{eq:assA2}.

If \eqref{eq:assA2FS} is assumed instead, we use
\eqref{eq:FSaR1} in every occurrence of
$\||x|^{1/2}a\,\partial R_{1}\phi\|_{\ell^{1}L^{2}}$.
These are precisely the $R_{1}$ pieces in the estimates of
$I_{1}$, $II_{1}$ and $II_{4}$; all terms containing $R_{2}$ or
$b$ are unchanged.  Thus \eqref{eq:finalest} is replaced by
\begin{equation}\label{eq:finalestFS}
  \|R(z)\phi\|_{L^{q,\infty}}\lesssim
  C_{FS}(a,b)|z|^{\mu_{2}}(1+|z|^{\sigma_{n}/2})
  \|\phi\|_{L^{p_{B},1}},
\end{equation}
where $C_{FS}(a,b)$ is controlled by the quantities in
\textbf{(H)} and by the norm in \eqref{eq:assA2FS}.  Since
$\mu_{2}=\theta(p_{B},q)$ and
$\sigma_{n}/2=(n-1)/(4n)$, this is the weak endpoint estimate
claimed in the Frank--Simon alternative.  The strong estimates
follow from it by the same duality and interpolation argument used
above.

\subsection{Proof of Theorem \ref{the:1}} \label{sub:the_1}

Writing as before $H=-\Delta+a \cdot \partial +b$
as in \eqref{eq:rewrH}, we use the standard decompositions
\begin{equation}\label{eq:dec2}
  R(z)=R_{0}-R_{0}(a \cdot \partial+b)R=
  R_{0}-R(a \cdot \partial +b)R_{0}.
\end{equation}
We shall first prove that $R(z)$ satisfies an estimate similar to
\eqref{eq:weiR0} (with $q=2$).
Recall that $\theta(p,2)=\frac n2(\frac 1p-\frac12)-1$.

\begin{lemma}[]\label{lem:gutR}
  Under Assumption \textbf{(H)}, the resolvent $R(z)=(H-z)^{-1}$
  satisfies for all $n\ge3$, $p\in(1,\infty)$ and
  $z\in\mathbb{C}\setminus[0,+\infty)$
  \begin{equation}\label{eq:gutR}
    \textstyle
    \||x|^{-\frac 12}R(z)\phi\| _{\ell^{\infty}L^{2}}
    \lesssim
    |z|^{\frac 14+\theta(p,2)}
    \|\phi\|_{L^{p}}
    \qquad\text{provided}\qquad 
    \frac 12+\frac{1}{n+1}\le \frac 1p\le \frac 12+\frac{3}{2n}.
  \end{equation}
\end{lemma}

\begin{proof}[Proof of the Lemma]
  We use the dual estimate of \eqref{eq:weiR0}, that is
  \begin{equation}\label{eq:dualR0}
    \|R_{0}(z)\phi\|_{L^{p'}}\lesssim
    |z|^{\frac 14+\theta(p,2)}
    \||x|^{\frac 12}\phi\| _{\ell^{1}L^{2}}
  \end{equation}
  with $p$ as in \eqref{eq:gutR}, and we shall prove that
  the dual estimate of \eqref{eq:gutR} is valid. Clearly,
  it is sufficient to estimate the second term in \eqref{eq:dec2}.
  Applying \eqref{eq:dualR0} we get
  \begin{equation*}
    |z|^{-\frac 14-\theta(p,2)}
    \|R_{0}(z)(a \cdot \partial+b)R\phi\|
      _{L^{p'}}\lesssim
    \||x|^{\frac 12}(a \cdot \partial+b)R\phi\|_{\ell^{1}L^{2}}
  \end{equation*}
  and by H\"{o}lder
  \begin{equation*}
    \le
    \||x|a\|_{\ell^{1}L^{\infty}}
      \||x|^{-\frac 12}\partial R \phi\|_{\ell^{\infty}L^{2}}
    +\||x|^{2}b\|_{\ell^{1}L^{\infty}}
      \||x|^{-\frac 32}R \phi\|_{\ell^{\infty}L^{2}}.
  \end{equation*}
  Both norms of $a$ and $b$ are finite thanks to \textbf{(H)}.
  Recalling the resolvent estimate \eqref{eq:resolvestR},
  we conclude that the right hand side 
  is bounded by $C\||x|^{\frac 12}\phi\|_{\ell^{1}L^{2}}$
  as claimed.
\end{proof}

In order to prove Theorem \ref{the:1} we shall use the second
decomposition in \eqref{eq:dec2}.
Since $R_{0}$ satisfies \eqref{eq:URER0}
for the closed region $DEE'=\Delta_{1}(n)\subset \Delta(n)$,
we focus on the piece $R(a \cdot \partial +b)R_{0}$;
by duality and interpolation, it is sufficient to prove
the estimate on the closed segment $DE$.
Taking $q\in[q_{D},q_{E}]$ and using the dual of 
\eqref{eq:gutR} we have
\begin{equation*}
  \textstyle
  \|Ra \cdot \partial R_{0}\phi\|_{L^{q}}\lesssim
  |z|^{\mu_{2}}
  \||x|^{\frac 12}a \cdot \partial R_{0}\phi\|_{\ell^{1}L^{2}},
  \qquad
  \mu_{2}=\frac n2(\frac 12-\frac 1q)-\frac 34.
\end{equation*}
Splitting $R_{0}=R_{1}+R_{2}$, we have by H\"{o}lder
and \eqref{eq:estR1gut} with $|\alpha|=1$
\begin{equation*}
  \||x|^{\frac 12}a \cdot \partial R_{1}\phi\|_{\ell^{1}L^{2}}
  \le
  \||x|a\|_{\ell^{1}L^{\infty}}
  \||x|^{-\frac 12}\partial R_{1}\phi\|_{\ell^{\infty}L^{2}}
  \lesssim
  |z|^{\frac{n}{2}(\frac1{p_E}-\frac12)-1+\frac12+\frac14}\|\phi\|_{L^{p_{E}}}
\end{equation*}
while by H\"{o}lder and by \eqref{eq:estR2gut} with $|\alpha|=1$,
$\nu=\frac{1}{n+1}$
\begin{equation*}
  \||x|^{\frac 12}a \cdot \partial R_{2}\phi\|_{\ell^{1}L^{2}}
  \le
  \||x|^{\frac 12+\nu}a\|_{\ell^{1}L^{\infty}}
  \||x|^{-\nu}\partial R_{2}\phi\|_{\ell^{\infty}L^{2}}
  \lesssim |z|^{\frac12+\frac{\nu}{2}+\theta(p_E,2)}
  \|\phi\|_{L^{p_{E}}}
\end{equation*}
Summing up we have on the line $DE$
\begin{equation}\label{eq:firstG}
  \|Ra \cdot \partial R_{0}\phi\|_{L^{q}}\lesssim
  |z|^{\frac n2(\frac 1{p_E}-\frac 1q)-1}
  (1+|z|^{-\frac 14+\frac{1}{2(n+1)}})
  \|\phi\|_{L^{p_{E}}}.
\end{equation}
Consider next the term $RbR_{0}$, which we write in the form
\begin{equation*}
  RbR_{0}=R_{0}bR_{0}+R_{0}bRbR_{0}+ R_{0}a \cdot\partial R bR_{0}.
\end{equation*}
We have already estimated the terms
$I_{2}=R_{0}bR_{0}$ and $II_{2}=R_{0}bRbR_{0}$ in the proof of 
Theorem \ref{the:2}, see \eqref{eq:estI2} and 
\eqref{eq:estII2}; thus we know that $I_{2},II_{2}$ are bounded
on the entire quadrilateral $ABB'A'$ with a power of $|z|$
equal to $|z|^{\theta(p,q)}$. Thus in particular on
the triangle $DEE'$ we get
\begin{equation*}
  \|(R_{0}bR_{0}+R_{0}bRbR_{0})\phi\|_{L^{q}}\lesssim
  C_{0}(b)|z|^{\theta(p_{E},q)}
  \|\phi\|_{L^{p_{E}}}
\end{equation*}
where 
$C_{0}(b)=\|b\|_{L^{\frac n2,1}}+
  \||x|^{\frac 32}b\|_{\ell^{p_{B}}L^{2n,2}}
  \||x|^{\frac 12}b\|_{\ell^{q_{A}'}L^{2n/3,2}}$.
On the other hand by the dual of \eqref{eq:weiR0}
\begin{equation*}
  \textstyle
  \|R_{0}a \cdot \partial R bR_{0}\phi\|_{L^{q}}
  \lesssim
  |z|^{\theta(2,q)+\frac 14}
  \||x|^{\frac 12}a \cdot \partial R bR_{0}\phi\|_{\ell^{1}L^{2}},
  \qquad
\end{equation*}
and by H\"{o}lder and the resolvent estimate \eqref{eq:resolvestR}
\begin{equation*}
  \lesssim
  |z|^{\theta(2,q)+\frac 14}\||x|a\|_{\ell^{1}L^{\infty}}
  \||x|^{\frac 12}bR_{0}\phi\|_{\ell^{1}L^{2}}.
\end{equation*}
To estimate the last norm, we split
\begin{equation*}
  bR_{0}=b\chi^{2}(x)R_{1}+b(1-\chi^{2}(x))R_{1}+bR_{2}=
    Z_{1}+Z_{2}+Z_{3}.
\end{equation*}
where $\chi$ is a cutoff equal to 1 near $x=0$.
For the first piece $Z_{1}=b\chi^{2}R_{1}$ we may write
\begin{equation*}
  \||x|^{\frac 12}b\chi^{2}R_{1}\phi\|_{\ell^{1}L^{2}}\le 
  \||x|^{2-\delta}b\chi\|_{\ell^{1}L^{\infty}}
  \||x|^{-\frac 32+\delta}\chi R_{1}\phi\|_{\ell^{\infty} L^{2}}.
\end{equation*}
By Hardy's inequality \eqref{eq:hardy2} in the Appendix
with $\sigma_{2}=0$, $\sigma_{1}=\frac 12-\delta$
with $\delta>0$ small, we get
\begin{equation*}
  \||x|^{-\frac 32+\delta}\chi(x)R_{1}\phi\|
    _{\ell^{\infty} L^{2}}\lesssim
  \||x|^{-\frac 12+\delta}\partial(\chi(x)R_{1}\phi)\|
    _{\ell^{\infty} L^{2}}
\end{equation*}
and using \eqref{eq:estR1gut} with $|\alpha|=1$ and 0 we conclude 
\begin{equation*}
  \|Z_{1}\phi\|_{\ell^{1}L^{2}}\lesssim
  \||x|^{2-\delta}b\chi\|_{\ell^{1}L^{\infty}}
  |z|^{-\frac{1}{2(n+1)}}(|z|^{\frac 14}+|z|^{-\frac 14})
  \|\phi\|_{L^{p_{E}}}.
\end{equation*}
For $Z_{2}=b(1-\chi^{2})R_{1}$ we have simply,
with $c_{b}=\||x|b(1-\chi^{2})\|_{\ell^{1}L^{\infty}}$
\begin{equation*}
  \||x|^{\frac 12}Z_{2}\phi\|_{\ell^{1}L^{2}}
  \lesssim
  c_{b}
  \||x|^{-\frac 12}R_{1}\phi\|_{\ell^{\infty}L^{2}}\lesssim
  c_{b}
  |z|^{\frac 14-\frac{1}{2(n+1)}}\|\phi\|_{L^{p_{E}}}.
\end{equation*}
For $Z_{3}=bR_{2}$ we use \eqref{eq:estR2gut} with $\alpha=0$
and $\nu'=1+\frac{1}{n+1}$:
\begin{equation*}
  \||x|^{\frac 12}bR_{2}\|_{\ell^{1}L^{2}}\le 
  \||x|^{\frac 12+\nu'}b\|_{\ell^{1}L^{\infty}}
  \||x|^{-\nu'}R_{2}\phi\|_{\ell^{\infty}L^{2}}
  \lesssim
  \||x|^{\frac 12+\nu'}b\|_{\ell^{1}L^{\infty}}
  |z|^{\frac{\nu'}{2}+\theta(p_{E},2)}
  \|\phi\|_{L^{p_{E}}}.
\end{equation*}
Summing all the terms we get,
since $\mu_{2}-\frac{1}{2(n+1)}=\theta(p_{E},q)-\frac 14$,
\begin{equation}\label{eq:bR0}
  \|R(z)\phi\|_{L^{q}}
  \lesssim
  C(a,b)
  |z|^{\theta(p_{E},q)-\frac 14}(|z|^{\frac 14}+|z|^{-\frac 14})
  \|\phi\|_{L^{p_{E}}}
\end{equation}
where $C(a,b)$ is the sum of the quantities
\begin{equation*}
  \|(|x|+|x|^{\frac 12+\frac{1}{n+1}})a\|_{\ell^{1}L^{\infty}},
  \qquad
  \|b\|_{L^{\frac n2,1}}+
    \||x|^{\frac 32}b\|_{\ell^{p_{B}}L^{2n,2}}
    \||x|^{\frac 12}b\|_{\ell^{q_{A}'}L^{2n/3,2}}
\end{equation*}
\begin{equation*}
  \||x|^{2-\delta}b\chi\|_{\ell^{1}L^{\infty}}+
  \||x|b(1-\chi^{2})\|_{\ell^{1}L^{\infty}}+
  \||x|^{\frac 32+\frac{1}{n+1}}b\|_{\ell^{1}L^{\infty}}
\end{equation*}
which are all finite by assumption \eqref{eq:assdec}
and \eqref{eq:assdec2}.
By duality and interpolation we obtain \eqref{eq:garcia}.

\subsection{Proof of Remark \ref{rem:potV1}}\label{sec:rem}

By a simple modification of the proof of Theorem 1.2 in
\cite{DAncona20}, one can extend Lemma \ref{lem:resolvestR}
to more general operators
\begin{equation*}
  \widetilde{H}=H+V_{1}=
  (i\partial+A(x))^{2}+V(x)+V_{1}(x)
\end{equation*}
provided $\widetilde{H}$ is selfadjoint nonnegative,
nonresonant at 0, $A,V$ satisfying \eqref{eq:VAass}
while $V_{1}$ is real valued and satisfies for some $\delta>0$
\begin{equation}\label{eq:assV1}
  |V_{1}(x)|\lesssim \bra{x}^{-1-\delta}.
\end{equation}
Under these assumptions, the resolvent 
$\widetilde{R}(z)=(H+V_{1}-z)^{-1}$ satisfies 
the following estimate: for every $\epsilon_{0}>0$
there exists $C(\epsilon_{0})$ such that,
for all $|z|\ge \epsilon_{0}$, $|\Im z|\le1$,
\begin{equation}\label{eq:resolvestRtil}
  \||x|^{-\frac32}\widetilde{R}(z)\phi\|_{\ell^{\infty}L^{2}}+
  |z|^{1/2} \||x|^{-\frac 12}\widetilde{R}(z)\phi\|
    _{\ell^{\infty}L^{2}}+
  \||x|^{-\frac 12}\partial \widetilde{R}(z)\phi\|
    _{\ell^{\infty}L^{2}}\le C(\epsilon_{0})
  \||x|^{\frac 12}\phi\|_{\ell^{1}L^{2}}
\end{equation}

To prove \eqref{eq:resolvestRtil}, 
we first note that Theorem 2.1 of 
\cite{DAncona20} (the estimate for large frequencies)
does not require any modification since the allowed potential
may already decay like $|x|^{-1-\delta}$ for large $|x|$.
Concerning the low frequency estimate,
it is sufficient to modify the argument of
Lemma 3.1 in \cite{DAncona20} as follows: using the notations
of that paper, the operator
$K(z)=(W+iA \cdot \partial+i \partial \cdot A)R_{0}(z)$
must be replaced by $\widetilde{K}(z)=K(z)+V_{1}R_{0}(z)$.
The new operator $\widetilde{K}(z)$ is also compact from
$\dot Y^{*}$ to $\dot Y^{*}$, where $\dot Y^{*}$ is the space
with norm $\||x|^{\frac 12}\phi\|_{\ell^{1}L^{2}}$,
thanks to the estimate
\begin{equation*}
  \|V_{1}R_{0}(z)v\|_{\dot Y^{*}}\le
  \||x|V_{1}\|_{\ell^{1}L^{\infty}}\|R_{0}(z)v\|_{\dot Y}\lesssim
  \||x|V_{1}\|_{\ell^{1}L^{\infty}} \cdot|z|^{-1/2}
  \|v\|_{\dot Y^{*}}.
\end{equation*}
Clearly, this estimate can be applied only for $z\neq0$.
Then the proof of Lemmas 3.1--3.3 in \cite{DAncona20}
follows with minimal modifications,
and we obtain \eqref{eq:resolvestR} with a constant uniform
on $|z|\ge \epsilon_{0}$ for every $\epsilon_{0}>0$.
Using \eqref{eq:resolvestRtil} instead of \eqref{eq:resolvestR}
in the proof of Theorem \ref{the:1}, we obtain \eqref{eq:garcia2}.

\section{Restriction type estimates}\label{sec:restr}

By the Stone formula, the density of the spectral measure
for the selfadjoint operator $H$ can be expressed as
\begin{equation*}
  E'_{H}(\lambda)=
  \frac{1}{2\pi i} (R(\lambda+i0)-R(\lambda-i0))=
  \frac{1}{\pi} \Im R(\lambda+i0),
  \qquad
  \lambda\in \mathbb{R}
\end{equation*}
regarded as the limit of 
$\frac{1}{2\pi i} (R(\lambda+i\epsilon)-R(\lambda-i\epsilon))$
as $\epsilon \downarrow0$.
The operator $E'_{H}(\lambda)$ can be written in terms of the
unperturbed spectral measure $E'_{-\Delta}$ as follows.
Denote by $K(z)$ the operator
\begin{equation*}
  K(z)=WR_0(z)
  \qquad\text{where}\quad 
  W=a\cdot\partial+b
\end{equation*}
with $a(x),b(x)$ as in \eqref{eq:rewrH}.
If $I-K(z)$ is invertible, we have the 
\emph{Lippmann-Schwinger} representation of $R(z)$
\begin{equation}
   R(z)=R_0(z)(I-K(z))^{-1}.
\end{equation}
In \cite{DAncona20} the operator $I-K(z)$ is studied on the
space $\dot Y$ and its dual $\dot Y^{*}$, with norms
\begin{equation*}
  \|\phi\|_{\dot Y}=\||x|^{-\frac 12}\phi\|_{\ell^{\infty}L^{2}},
  \qquad
  \|\phi\|_{\dot Y^{*}}=\||x|^{\frac 12}\phi\|_{\ell^{1}L^{2}}.
\end{equation*}
In particular, in Theorem 3.5. of \cite{DAncona20} it is proved 
that under Assumption \textbf{(H)} the operator
$I-K(z)$ is bounded and invertible on $\dot Y^{*}$,
with $(I-K(z))^{-1}$ bounded uniformly for $z$ in bounded
subsets of $\{\Im z\ge0\}$ and of $\{\Im z\le0\}$.
This ensures that the limits $R(\lambda\pm i0)$ are well
defined as bounded operators from $\dot Y$ to $\dot Y^{*}$
(as implied by the resolvent estimate \eqref{eq:resolvestR}).

Starting from the trivial identity
\begin{equation*}
  (I+R_{0}(\overline{z})W)R(z)-R_{0}(\overline{z})=
  R(z)-R_{0}(\overline{z})(I-WR(z))
\end{equation*}
and using the Lippmann--Schwinger relations
\begin{equation*}
  R_{0}(\overline{z})=(I+R_{0}(\overline{z})W)R(\overline{z}),
  \qquad
  R(z)=R_{0}(z)(I-WR(z)),
\end{equation*}
we get
\begin{equation*}
  (I+R_{0}(\overline{z})W)(R(z)-R(\overline{z}))=
  (R_{0}(z)-R_{0}(\overline{z}))(I-WR(z))
\end{equation*}
which implies (see \cite{Mizutani20-b})
\begin{equation*}
  R(z)-R(\overline{z})=
  (I+R_{0}(\overline{z})W)^{-1}(R_{0}(z)-R_{0}(\overline{z}))
  (I-WR(z)).
\end{equation*}
Taking $z=\lambda+i \epsilon$ and letting $\epsilon \downarrow0$
we obtain the identity
\begin{equation}\label{eq:idEl}
  E'_{H}(\lambda)=
  (I+R_0(\lambda-i0)W)^{-1}E'_{-\Delta}(\lambda)
  (I-WR(\lambda+i0))
\end{equation}
which we shall use in the following.

\begin{theorem}
  Suppose $H$ satisfies the assumptions of Theorem \ref{the:2}
  and in addition
  $a|x|^{1/2}\in \ell^{\frac{2n}{n+3}}L^{\frac{2n}{3}}$.
  Then for any $\lambda>0$ we have the estimate
  \begin{equation}\label{eq:Restri}
    \textstyle
    \|E'_H(\lambda)\phi\|_{L^{p'}}
    \leq C (\lambda)\|\phi\|_{L^{p}},
    \qquad
    \frac 12+\frac{1}{n+1}\le\frac 1p\le \frac 12+\frac3{2n}
  \end{equation}
\end{theorem}

\begin{proof}
  In the range
  \begin{equation*}
    \textstyle
    \frac 12+\frac{1}{n+1}\le\frac 1p\le \frac 12+\frac1n
  \end{equation*}
  estimate \eqref{eq:Restri} is a direct consequence of the
  definition 
  $E'_{H}(\lambda)=\frac 1\pi\Im R(\lambda+i0)$ and of
  \eqref{eq:UREH}. To obtain the full range \eqref{eq:Restri},
  by interpolation,
  it is then sufficient to prove the estimate for
  the endpoint value $p=\frac{2n}{n+3}$.
  Recalling the estimate for the free case \eqref{eq:restrfree}
  and the representation \eqref{eq:idEl}, we must only prove that
  $WR(\lambda+i0)$ extends to a bounded operator on
  $L^{\frac{2n}{n+3}}$ while $(I+R_{0}(\lambda-i0)W)^{-1}$
  extends to a bounded operator on $L^{\frac{2n}{n-3}}$.
  Here prove the first fact, while the second will
  be proved in the following Lemma.

  We split $WR(\lambda+i0)$ into three parts
  \begin{equation*}
     WR=a\cdot\partial R^1+a\cdot\partial R^2+bR
  \end{equation*}
  where, writing for simplicity $z=\lambda+i0$,
  \begin{equation}\label{eq:decomR}
    R^{1}=R_1(I-K(z))^{-1}, \quad
    R^{2}=R_2(I-K(z))^{-1}.
  \end{equation} 
  By Lemma 3.3 in \cite{DAncona20}, the operator 
  $(I-K(z))^{-1}$ is bounded on $Y^*$.
  Using the dual of \eqref{eq:estR1gut} we can write for
  $\frac{1}{2}+\frac{1}{n+1}\le\frac1p<1$ 
  and $\mu=\frac{|\alpha|}{2}+\frac 14+\theta(p,2)$
  \begin{align}\label{eq:estR1}
  \begin{split}
      \|\partial^\alpha R^1 \phi\|_{L^{p'}}&= 
      \|\partial R_1(I-K(z))^{-1}f\|_{L^{p'}}\\
    &\lesssim |z|^{\mu}
      \||x|^{\frac 12}(I-K(z))^{-1}\phi\|_{\ell^{1} L^2}\\
    &\lesssim |z|^{\mu}C(z)
      \||x|^{\frac 12}\phi\|_{\ell^{1} L^2}
  \end{split}
  \end{align}
  where $C(z)$ is the norm of $(I-K(z))^{-1}$ on $\dot Y^{*}$,
  while using the dual of \eqref{eq:estR2gut} with $\nu=\frac 12$
  we can write for
  $\frac{1}{2}-\frac{1}{2n}\leq 
    \frac{1}{p}\leq\frac{1}{2}+\frac{3}{2n}
    -\frac{|\alpha|}{n}$ 
  \begin{align}\label{eq:estR2}
  \begin{split}
     \|\partial^\alpha R^2\phi\|_{L^{p'}}&= 
     \|\partial R_2(I-K(z))^{-1}\phi\|_{L^{p'}}\\
    &\lesssim |z|^{\mu}
    \||x|^{\frac12}(I-K(z))^{-1}\phi\|_{\ell^{1} L^2}\\
    &\lesssim |z|^{\mu}C(z)
    \||x|^{\frac12}\phi\|_{\ell^{1} L^2}.
  \end{split}
  \end{align}
  Then by H\"older inequality and the dual of \eqref{eq:estR1} 
  we get
  \begin{equation*}
    \|a\cdot\partial R^1 \phi\|_{L^{\frac{2n}{n+3}}}\le 
    \|a|x|^{\frac 12}\|_{\ell^{\frac{2n}{n+3}}L^{\frac{2n}{3}}} 
    \||x|^{-\frac 12}\partial R^1\phi\|_{\ell^\infty L^2}
    \lesssim
    \|a|x|^{\frac 12}\|_{\ell^{\frac{2n}{n+3}}L^{\frac{2n}{3}}} 
    |z|^{\mu}\|\phi\|_{L^{\frac{2n}{n+3}}}.
  \end{equation*}
  For the term $a\cdot\partial R^2$, by Lemma \ref{B:I+R0W} 
  below we know that $(I-K(z))^{-1}$
  is a bounded operator on $L^{\frac{2n}{n+3}}$,
  with norm $C(z)$,
  therefore we can write, using \eqref{eq:estR2sob},
  \begin{equation}
    \begin{aligned}
     \|a\cdot\partial R^2 \phi\|_{L^{\frac{2n}{n+3}}}&
     \le \|a\|_{L^{n}}\|\partial R^2 \phi\|_{L^{\frac{2n}{n+1}}}
     = \|a\|_{L^{n}}
     \|\partial R_2(I-K(z))^{-1}\phi\|_{L^{\frac{2n}{n+1}}}\\
     &\lesssim  
     \|a\|_{L^{n}}\|(I-K(z))^{-1}\phi\|_{L^{\frac{2n}{n+3}}}\\
     &\le  \|a\|_{L^{n}}C(z)\|\phi\|_{L^{\frac{2n}{n+3}}}.
     \end{aligned}
  \end{equation}
  For the last term we have, using \eqref{eq:gutR},
  \begin{equation}
  \begin{aligned}
     \|bR \phi\|_{L^{\frac{2n}{n+3}}}&\le 
     \|b|x|^{1/2} \|_{\ell^{\frac{2n}{n+3}}L^{\frac{2n}{3}}}
     \||x|^{-1/2}R \phi\|_{\ell^\infty L^2}\\
     &\leq \|b|x|^{1/2} \|_{\ell^{\frac{2n}{n+3}}L^{\frac{2n}{3}}}
     \|\phi\|_{L^{\frac{2n}{n+3}}}.
  \end{aligned}
  \end{equation}
  Thus we see that $WR$ is a bounded operator on
  $L^{\frac{2n}{n+3}}$ as claimed.
  Combining this estimate with the following Lemma, the proof
  is concluded.
\end{proof}

\begin{lemma}\label{B:I+R0W}
  Let $W=a\cdot\partial +b$ such that 
  $b\in L^{\frac n2}\cap L^{\frac n2 +\epsilon}$,
  $a\in L^{n}\cap L^{n+\epsilon}$ and
  $a|x|^{1/2}\in \ell^{\frac{2n}{n+3}}L^{\frac{2n}{3}}$ for some
  $\epsilon>0$. Then the operator
  $I+R_0(z)W$ is bounded and invertible on $L^{\frac{2n}{n-3}}$
  for all $z\neq 0$, and $z \mapsto(I+R_0(z)W)^{-1}$ is a
  continuous map in the operator norm
  for $z$ in the upper (resp.~lower) complex half plane
  minus the origin.
\end{lemma}

\begin{proof}[Proof of Lemma \ref{B:I+R0W}]
We shall prove that $R_{0}W$ 
is a compact operator on 
$L^{\frac{2n}{n-3}}$ and use Fredholm theory.
Equivalently, we can prove that 
$WR_{0}=a \cdot \partial R_{0}+bR_{0}$ is compact on 
$L^{\frac{2n}{n+3}}$.

To prove compactness of $a \cdot\partial R_{0}$ we split
$R_0(z)=R_1(z)+R_2(z)$ and handle the two terms separately.
By Sobolev embedding, the operator $\partial R_{2}$ is bounded
from $L^{\frac{2n}{n+3}}$ into $L^{\frac{2n}{n+1}}$ and
by compact embedding it is compact from $L^{\frac{2n}{n+3}}$ 
into $L^{\frac{2n}{n+1}-\epsilon}(|x|<R)$ for any
$\epsilon>0$ small and any $R>0$. 
If $a\in L^{n+\epsilon}(|x|<R)$ for some $\epsilon>0$
by H\"{o}lder inequality this implies that, for $z\neq0$,
\begin{equation*}
    a \cdot \partial R_{2}:L^{\frac{2n}{n+3}}\to
    L^{\frac{2n}{n+3}}(|x|<R)
\end{equation*}
is a compact operator.
On the other hand we can write by Sobolev embedding
\begin{equation*}
   \|a \cdot\partial R_{2}\phi\|_{L^{\frac{2n}{n+3}}(|x|>R)}\le
   \|a\|_{L^{n}(|x|>R)}
   \|\partial R_{2}\phi\|_{L^{\frac{2n}{n+1}}}
   \lesssim
   \epsilon(R)
   \|\phi\|_{L^{\frac{2n}{n+3}}}
\end{equation*} 
where $ \epsilon(R)=\|a\|_{L^{n}(|x|>R)}\to0$ as $R\to+\infty$.
Thus if $a\in L^{n}\cap L^{n+\epsilon}$ for some $\epsilon>0$
we conclude by a diagonal procedure that $a \cdot \partial R_{2}$
is a compact operator on $L^{\frac{2n}{n+3}}$.
Exactly the same argument shows that $bR_{2}$ is compact on
$L^{\frac{2n}{n+3}}$ provided 
$b\in L^{\frac n2}\cap L^{\frac n2 +\epsilon}$.

Consider now $a \cdot \partial R_{1}$. We can write
$\partial R_{1}=\psi(D)\partial R_{1}$
for a suitable $\psi\in \mathscr{S}$ so that
\begin{equation*}
    a \cdot\partial R_{1}=
    a|x|^{1/2}\cdot |x|^{-1/2}\psi(D)|x|^{1/2}\cdot
    |x|^{-1/2}\partial R_{1}
\end{equation*}
We see that
$a|x|^{1/2}$ is bounded from 
$\ell^{\infty}L^{2}(|x|<R)$ to $L^{\frac{2n}{n+3}}(|x|<R)$
provided $a$ satisfies 
$a\in \ell^{\frac{2n}{n+3}}L^{\frac{2n}{3}}$.
By Lemma \ref{lem:berndyad}, $|x|^{-1/2}\psi(D)|x|^{1/2}$ 
is a compact operator
from $\ell^{\infty}L^{2}$ to $\ell^{\infty}L^{2}(|x|<R)$. 
Finally, by estimate \eqref{eq:estR1gut} 
$|x|^{-1/2}\partial R_{1}$ is bounded from
$\ell^{\infty}L^{2}$ to $L^\frac{2n}{n+3}$.
We conclude that
\begin{equation*}
  a \cdot\partial R_{1}:L^\frac{2n}{n+3}\to
  L^\frac{2n}{n+3}(|x|<R)
\end{equation*}
is a compact operator. Since we have
\begin{equation*}
  \|a \cdot\partial R_{1}\phi\|_{L^{\frac{2n}{n+3}}(|x|>R)}\le
   \||x|^{1/2}a\|_{\ell^{\frac{2n}{n+3}}L^{\frac{2n}{3}}}
   \||x|^{-1/2}\partial R_{1}\phi\|_{\ell^{\infty}L^{2}}\le
   \epsilon(R)
   \|\phi\|_{L^{\frac{2n}{n+3}}}
\end{equation*}
and $\epsilon(R)\to 0$ as $R\to+\infty$, we see
that $a \cdot \partial R_{1}$ is a compact operator
on $L^{\frac{2n}{n+3}}$.
By a similar but simpler argument one checks that
$bR_{1}$ is also a compact operator on the same space provided
$b\in L^{\frac n2}$.

Summing all the pieces, we have proved that $R_{0}W$ is a compact
operator on $L^{\frac{2n}{n-3}}$ for all $z\neq0$.
The same estimates show that $z \mapsto R_{0}W$
is continuous in the operator norm of bounded operators on 
$L^{\frac{2n}{n-3}}$ on the closed upper and lower complex planes,
minus the origin. Since by assumption $I+R_{0}W$
is an injective operator, Fredholm theory ensures that
$I+R_{0}W$ is invertible with bounded inverse, and
$z \mapsto (I+R_{0}W)^{-1}$ is a continuous map in the
operator norm (see e.g.~Lemma 3.4 in \cite{DAncona20}).
\end{proof}

\section{Appendix}\label{sec:hardy}

We prove three lemmas here. The first one is a uniform 
estimate for the integral operator
\begin{equation*}
  T_{\mu}f(r)
  =
  A_{\mu}(r)\int_r^\infty B_{\mu}(s)f(s) s^{n-1}ds
  +
  B_{\mu}(r)\int_0^r A_{\mu}(s)f(s) s^{n-1}ds
\end{equation*}
where
\begin{equation*}
  A_{\mu}(r)=r^{-\frac{n-2}{2}}J_{\mu}(r),
  \qquad
  B_{\mu}(r)=r^{-\frac{n-2}{2}}H_{\mu}^{(1)}(r).
\end{equation*}
In the proof we use ideas from \cite{BarceloRuizVega97-a}
and \cite{FrankSimon17-a}. 

\begin{lemma}\label{lem:weak-scalar}
  We have, uniformly for $\mu\ge(n-2)/2$,
  \begin{equation}\label{eq:scalar-endpoint}
    \|T_{\mu}f\|_{L^{\frac{2n}{n-1},\infty}(r^{n-1}dr)}
    \lesssim
    \|f\|_{L^{\frac{2n}{n+1},1}(r^{n-1}dr)} .
  \end{equation}
\end{lemma}

\begin{proof}
  Write for brevity
  \begin{equation*}
  \textstyle
    p=\frac{2n}{n+1},
    \qquad
    q=p'=\frac{2n}{n-1},
    \qquad
    d\nu(r)=r^{n-1}\,dr.
  \end{equation*}
  It is enough to estimate the Volterra half
  \begin{equation*}
    T_\mu^+f(r)
    =
    A_\mu(r)\int_r^\infty B_\mu(s)f(s)\,d\nu(s).
  \end{equation*}
  Indeed, the second
  term in $T_{\mu}$ is the transpose of $T^{+}_{\mu}$, 
  up to complex conjugation of the Hankel function; since
  $|\overline{H_\mu^{(1)}}|=|H_\mu^{(1)}|$, the same estimate
  applies and Lorentz duality gives the bound for the transposed
  operator.  For $T_\mu^+$, Lorentz duality gives
  \begin{equation*}
    \left|\int_r^\infty B_{\mu}(s)f(s)\,d\nu(s)\right|
    \lesssim
    M_{\mu}^{+}(r)\|f\|_{L^{p,1}(d\nu)}
  \end{equation*}
  where $M_{\mu}^{+}$ is the outgoing tail function
  \begin{equation*}
    M_{\mu}^{+}(r)
    =
    \|B_{\mu}\mathbf{1}_{(r,\infty)}\|_{L^{q,\infty}(d\nu)} .
  \end{equation*}
  Hence it remains to prove
  \begin{equation}\label{eq:tail-plus}
    \sup_{\mu\ge(n-2)/2}
      \|A_{\mu}M_{\mu}^{+}\|_{L^{q,\infty}(d\nu)}
    <\infty .
  \end{equation}

  The standard asymptotics 
  $J_{\mu}\simeq_{\mu} r^{\mu}$, $H_{\mu}\simeq_{\mu} r^{-\mu}$
  (with $H_{0}\simeq \log r$) for small $r$,
  and $|J_{\mu}|+|H_{\mu}|\lesssim_{\mu} r^{-1/2}$ for large $r$,
  are sufficient to obtain \eqref{eq:tail-plus}
  for \emph{bounded} $\mu$.
  The problem is to get a uniform bound in $\mu$ for large $\mu$.
  Thus we can assume $\mu\ge \frac 12$, and in addition
  that $\mu$ is large enough for the intervals below to have the
  indicated order.  Set
  \begin{equation*}
    \rho=n-1-\frac{q(n-2)}2=\frac1{n-1}.
  \end{equation*}
  This is precisely the borderline value for which
  \begin{equation*}
    \frac q2=\rho+1 .
  \end{equation*}
  Besides standard estimates for Bessel functions,
  we use the sharper estimates from
  Proposition~A.2 in \cite{FrankSimon17-a},
  which quotes Barcelo--Ruiz--Vega \cite{BarceloRuizVega97-a}.
  Let $\alpha_0\in(0,1/2)$ be the constant in that Proposition,
  and decompose $(0,\infty)$ into
  \begin{align*}
    I_0&=(0,1),\\
    I_1&=(1,\mu\,\mathrm{sech}\,\alpha_0),\\
    I_2&=(\mu\,\mathrm{sech}\,\alpha_0,
      \mu-\mu^{1/3}),\\
    I_3&=(\mu-\mu^{1/3},\mu+\mu^{1/3}),\\
    I_4&=(\mu+\mu^{1/3},\infty).
  \end{align*}
  Put $R_\mu=\mu-\mu^{1/3}$.

  We first record the estimate obtained
  from the proof of \cite{FrankSimon17-a}, Lemma~A.3:
  \begin{equation}\label{eq:single-preturn}
    \|A_\mu\mathbf{1}_{I_1\cup I_2}\|_{L^{q,\infty}(d\nu)}
    \le C.
  \end{equation}
  On the turning interval $I_3$, one has
  \begin{equation*}
    |J_\mu(r)|+|H_\mu(r)|\le C\mu^{-1/3},
    \qquad |r-\mu|\le\mu^{1/3}.
  \end{equation*}
  Thus $|A_\mu|+|B_\mu|\lesssim
  \mu^{-(n-2)/2-1/3}$ on an interval of $d\nu$--measure
  $\simeq\mu^{n-2/3}$, and hence
  \begin{equation}\label{eq:single-turning}
    \|A_\mu\mathbf{1}_{I_3}\|_{L^{q,\infty}(d\nu)}
    +
    \|B_\mu\mathbf{1}_{I_3}\|_{L^{q,\infty}(d\nu)}
    \le C.
  \end{equation}
  On $I_4$ we use the standard oscillatory bound
  \begin{equation}\label{eq:postturn-bessel}
    |J_\mu(r)|+|H_\mu(r)|
    \le
    C r^{-1/4}(r-\mu)^{-1/4}.
  \end{equation}
  Split
  \begin{equation*}
    I_4=I_4^{\mathrm{near}}\cup I_4^{\mathrm{far}},
    \qquad
    I_4^{\mathrm{near}}=(\mu+\mu^{1/3},2\mu),
    \qquad
    I_4^{\mathrm{far}}=(2\mu,\infty).
  \end{equation*}
  On $I_4^{\mathrm{far}}$ one has $r-\mu\simeq r$, hence
  \begin{equation*}
    |A_\mu(r)|+|B_\mu(r)|
    \lesssim r^{-\frac{n-1}{2}} .
  \end{equation*}
  Since $\frac{n-1}{2}q=n$,
  \begin{equation*}
    \left\|
      r^{-\frac{n-1}{2}}\mathbf{1}_{(1,\infty)}
    \right\|_{L^{q,\infty}(r^{n-1}dr)}
    \le C.
  \end{equation*}
  The same bound then holds on any smaller tail $(R,\infty)$,
  $R\ge1$, by restriction.  On $I_4^{\mathrm{near}}$ write
  $\tau=r-\mu$.  Then
  $r\simeq\mu$, $d\nu(r)\simeq\mu^{n-1}d\tau$, and
  \eqref{eq:postturn-bessel} gives
  \begin{equation*}
    |A_\mu(r)|+|B_\mu(r)|
    \lesssim
    \mu^{-\frac{n-2}{2}-\frac14}\tau^{-1/4},
    \qquad
    \mu^{1/3}<\tau<\mu .
  \end{equation*}
  Therefore, using $q\le4$,
  \begin{align*}
    \left\|
      (|A_\mu|+|B_\mu|)
      \mathbf{1}_{I_4^{\mathrm{near}}}
    \right\|_{L^{q,\infty}(d\nu)}
    &\lesssim
    \mu^{-\frac{n-2}{2}-\frac14}
    \mu^{\frac{n-1}{q}}
    \left\|
      \tau^{-1/4}\mathbf{1}_{(\mu^{1/3},\mu)}
    \right\|_{L^{q,\infty}(d\tau)}\\
    &\lesssim
    \mu^{-\frac{n-2}{2}-\frac14}
    \mu^{\frac{n-1}{q}}
    \mu^{\frac1q-\frac14}
    =
    \mu^{-\frac{n-1}{2}+\frac nq}
    =1 .
  \end{align*}
  Hence
  \begin{equation}\label{eq:weak-postturn}
    \|A_\mu\mathbf{1}_{I_4}\|_{L^{q,\infty}(d\nu)}
    +
    \|B_\mu\mathbf{1}_{I_4}\|_{L^{q,\infty}(d\nu)}
    \le C .
  \end{equation}
  In particular, we have proved that
  \begin{equation}\label{eq:B-tail-after-R}
    \|B_\mu\mathbf{1}_{(R_\mu,\infty)}\|
      _{L^{q,\infty}(d\nu)}
    \le C .
  \end{equation}

  We now estimate $A_\mu M_\mu^+$ on the output intervals.
  On $I_0$ the bounds
  \begin{equation*}
    |J_\mu(r)|
    \le C\frac{(r/2)^\mu}{\Gamma(\mu+1)},
    \qquad
    |H_\mu^{(1)}(r)|
    \le C\frac{\Gamma(\mu)}{(r/2)^\mu}
  \end{equation*}
  control the part of $M_\mu^+$ in $(r,1)$; the remaining tail
  must be estimated separately.  We use the classical
  monotonicity for $\mu\ge1/2$ of
  \begin{equation*}
    s\mapsto s|H_\mu^{(1)}(s)|^2
  \end{equation*}
  which is nonincreasing.  Thus for $s\ge1$,
  \begin{equation*}
    |H_\mu^{(1)}(s)|
    \le s^{-1/2}|H_\mu^{(1)}(1)|
    \le C2^\mu\Gamma(\mu)s^{-1/2}.
  \end{equation*}
  Recalling that
  $\frac nq=\frac{n-1}{2}$ and  $r^{-\mu+1/2}\ge1$
  for $0<r\le1$, it follows that
  \begin{align*}
    M_\mu^+(r)
    &\le
    C2^\mu\Gamma(\mu)
    \|
      s^{-\mu-\frac{n-2}{2}}\mathbf{1}_{(r,1)}
    \|_{L^{q,\infty}(d\nu)}
    +
    C2^\mu\Gamma(\mu)
    \|
      s^{-\frac{n-1}{2}}\mathbf{1}_{(1,\infty)}
    \|_{L^{q,\infty}(d\nu)}\\
    &\le
    C\,2^\mu\Gamma(\mu)\,r^{-\mu+\frac12}.
  \end{align*}
  Multiplying by $A_\mu(r)$ gives
  \begin{equation}\label{eq:origin-model}
    |A_\mu(r)|M_\mu^+(r)
    \le C r^{-\frac{n-3}{2}}
  \end{equation}
  and the right hand side belongs to $L^{q,\infty}((0,1),d\nu)$,
  uniformly in $\mu$.

  We next estimate $A_\mu M_\mu^+$ on $I_1\cup I_2$. Put
  \begin{equation*}
    N_\mu^+(r)=
    \|B_\mu\mathbf{1}_{(r,R_\mu)}\|_{L^{q,\infty}(d\nu)} .
  \end{equation*}
  The tail of $M_\mu^+$ beyond $R_\mu$ is harmless by
  \eqref{eq:B-tail-after-R} and \eqref{eq:single-preturn}.
  Thus it remains to estimate $A_\mu N_\mu^+$.
  Monotonicity of $s|H_\mu(s)|^2$ gives, for $s\ge r$,
  \begin{equation}\label{eq:Hmon-tail}
    |B_\mu(s)|
    \le
    r^{1/2}|H_\mu(r)|s^{-\frac{n-1}{2}} .
  \end{equation}
  On $I_1$ this immediately implies
  $N_\mu^+(r)\le C r^{1/2}|H_\mu(r)|$, since
  $s^{-(n-1)/2}$ has bounded weak $L^q(d\nu)$ norm on every
  tail.  Using the product estimate from
  Proposition~A.2 in \cite{FrankSimon17-a},
  \begin{equation*}
    |J_\mu(r)H_\mu(r)|\le C\mu^{-1},
    \qquad r\le\mu\operatorname{sech}\alpha_0,
  \end{equation*}
  we obtain
  \begin{equation*}
    |A_\mu(r)|N_\mu^+(r)
    \le
    C\mu^{-1}r^{-\frac{n-3}{2}},
    \qquad r\in I_1.
  \end{equation*}
  Consequently
  \begin{equation*}
    \|A_\mu N_\mu^+\mathbf{1}_{I_1}\|_{L^{q,\infty}(d\nu)}
    \le
    C\mu^{-1}
    \left(\int_1^{C\mu}
      r^{n-1-\frac{q(n-3)}2}\,dr
    \right)^{1/q}
    \le C,
  \end{equation*}
  because $n-q(n-3)/2=q$.

  On $I_2$ the crude tail bound above loses a power. One must
  also use the finite length of the remaining pre--turning tail.
  Write $L=\mu-r$.  If $r<s<R_\mu$, then $r\simeq s\simeq\mu$
  and $s-r\le L$, so \eqref{eq:Hmon-tail} gives
  \begin{equation*}
    N_\mu^+(r)
    \le
    C\mu^{-\frac{n-2}{2}}
      \mu^{\frac{n-1}{q}}L^{1/q}|H_\mu(r)| .
  \end{equation*}
  The transition product estimate
  \begin{equation*}
    |J_\mu(r)H_\mu(r)|
    \le
    C\mu^{-1/2}(\mu-r)^{-1/2},
    \qquad
    \mu\operatorname{sech}\alpha_0<r<R_\mu,
  \end{equation*}
  therefore yields
  \begin{equation*}
    |A_\mu(r)|N_\mu^+(r)
    \le
    C\mu^{-n+\frac32+\frac{n-1}{q}}
    L^{\frac1q-\frac12}.
  \end{equation*}
  Since $d\nu(r)\simeq\mu^{n-1}dL$ on $I_2$ and
  $q\le4$,
  \begin{align*}
    \|A_\mu N_\mu^+\mathbf{1}_{I_2}\|_{L^{q,\infty}(d\nu)}
    &\le
    C
    \mu^{-n+\frac32+\frac{n-1}{q}}
    \mu^{\frac{n-1}{q}}
    \left\|
      L^{\frac1q-\frac12}
      \mathbf{1}_{(\mu^{1/3},C\mu)}
    \right\|_{L^{q,\infty}(dL)}\\
    &\le
    C
    \mu^{-n+1+\frac{2n}{q}}
    =C .
  \end{align*}
  This proves the pre--turning contribution to
  \eqref{eq:tail-plus}.

  In order to estimate $A_\mu M_\mu^+$ on $I_3$, 
  no monotonicity is needed and
  we use only the single function bounds already recorded.
  If $r\in I_3$, then
  \begin{equation*}
    M_\mu^+(r)
    \le
    \|B_\mu\mathbf{1}_{I_3}\|_{L^{q,\infty}(d\nu)}
    +
    \|B_\mu\mathbf{1}_{I_4}\|_{L^{q,\infty}(d\nu)}
    \le C
  \end{equation*}
  by \eqref{eq:single-turning} and \eqref{eq:weak-postturn}.
  Therefore, by \eqref{eq:single-turning},
  \begin{equation*}
    \|A_\mu M_\mu^+\mathbf{1}_{I_3}\|_{L^{q,\infty}(d\nu)}.
    \le C
  \end{equation*}

  Finally, to estimate $A_\mu M_\mu^+$ on $r\in I_4$, we write
  \begin{equation*}
    M_\mu^+(r)
    \le
    \|B_\mu\mathbf{1}_{I_4}\|_{L^{q,\infty}(d\nu)}
    \le C,
  \end{equation*}
  and hence, by \eqref{eq:weak-postturn},
  \begin{equation*}
    \|A_\mu M_\mu^+\mathbf{1}_{I_4}\|_{L^{q,\infty}(d\nu)}.
    \le C
  \end{equation*}
  Combining \eqref{eq:origin-model}, the pre--turning estimates,
  \eqref{eq:single-turning}, and \eqref{eq:weak-postturn} proves
  \eqref{eq:tail-plus}, and this concludes the proof.
\end{proof}

The second lemma is a simple tool to handle mixed norms of
spherical harmonics decompositions:

\begin{lemma}[Hilbert-valued Lorentz summation]\label{lem:hilbert}
  Let $1<p\le2$, $q=p'$, and let $S_j$ be linear operators
  satisfying
  \begin{equation*}
    \|S_j f\|_{L^{q,\infty}(Y)}
    \le C\|f\|_{L^{p,1}(X)}
  \end{equation*}
  with $C$ independent of $j$.  Then the diagonal operator
  $S(\{f_j\})=\{S_jf_j\}$ satisfies
  \begin{equation*}
    \bigl\|
      \bigl(\sum_j |S_jf_j|^2\bigr)^{1/2}
    \bigr\|_{L^{q,\infty}(Y)}
    \le
    C_p C
    \bigl\|
      \bigl(\sum_j |f_j|^2\bigr)^{1/2}
    \bigr\|_{L^{p,1}(X)} .
  \end{equation*}
\end{lemma}

\begin{proof}
  By Lorentz duality, each $S_j$ satisfies
  \begin{equation*}
    |\langle S_j f,g\rangle|
    \le
    C\|f\|_{L^{p,1}(X)}\|g\|_{L^{p,1}(Y)} .
  \end{equation*}
  Therefore, for finite sequences,
  \begin{align*}
    \bigl|\sum_j\langle S_jf_j,g_j\rangle\bigr|
    &\le
    C
    \bigl(\sum_j\|f_j\|_{L^{p,1}}^2\bigr)^{1/2}
    \bigl(\sum_j\|g_j\|_{L^{p,1}}^2\bigr)^{1/2}.
  \end{align*}
  Since $1<p\le2$, the Lorentz space $L^{p,1}$ is
  $p$-convex, and in particular
  \begin{equation*}
    \bigl(\sum_j\|f_j\|_{L^{p,1}}^2\bigr)^{1/2}
    \le
    C_p
    \bigl\|
      \bigl(\sum_j|f_j|^2\bigr)^{1/2}
    \bigr\|_{L^{p,1}} .
  \end{equation*}
  The same estimate applies to $g_j$. A second use of Lorentz
  duality for the corresponding $\ell^2$-valued spaces gives the
  claim.
\end{proof}

The last lemma is a Hardy type estimate in weighted dyadic spaces.

\begin{lemma}[]\label{lem:hardy}
  Let $n\ge3$, $\sigma_{1}<\frac n2-1$, and
  $2\sigma_{2}<n-2-2\sigma_{1}$.
  Then we have the estimates
  \begin{equation}\label{eq:hardy1}
    \||x|^{-\sigma_{1}-1}\bra{x}^{-\sigma_{2}}\phi\|_{L^{2}}
    \lesssim
    \||x|^{-\sigma_{1}}\bra{x}^{-\sigma_{2}}\partial \phi\|_{L^{2}},
  \end{equation}
  \begin{equation}\label{eq:hardy2}
    \||x|^{-\sigma_{1}-1}\bra{x}^{-\sigma_{2}}\phi\|
      _{\ell^{\infty}L^{2}}
    \lesssim
    \||x|^{-\sigma_{1}}\bra{x}^{-\sigma_{2}}\partial \phi\|
      _{\ell^{\infty}L^{2}}.
  \end{equation}
\end{lemma}

\begin{proof}
  We recall the well known Hardy estimate
  \begin{equation*}
    \textstyle
    \||x|^{-\sigma_{1}-1}u\|_{L^{2}}\le
    \frac{2}{n-2-2 \sigma_{1}}\||x|^{-\sigma_{1}}\partial u\|_{L^{2}},
    \qquad
    \sigma_{1}<\frac n2-1.
  \end{equation*}
  We apply this to $u=\bra{x}^{-\sigma_{2}}\phi$; since
  \begin{equation*}
    \textstyle
    |\partial u|=|\bra{x}^{-\sigma_{2}}\partial\phi-
      \sigma_{2} \bra{x}^{-\sigma_{2}-2}x \phi|\le 
    \bra{x}^{-\sigma_{2}}|\partial\phi|+
    \sigma_{2} \bra{x}^{-\sigma_{2}}|x|^{-1}|\phi|,
  \end{equation*}
  we obtain
  \begin{equation*}
    \textstyle
    \||x|^{-\sigma_{1}-1}\bra{x}^{-\sigma_{2}}\phi\|_{L^{2}}\le
    \frac{2}{n-2-2 \sigma_{1}}
    \||x|^{-\sigma_{1}}\bra{x}^{-\sigma_{2}}\partial \phi\|_{L^{2}}+
    \frac{2 \sigma_{2}}{n-2-2 \sigma_{1}}
    \||x|^{-\sigma_{1}-1}\bra{x}^{-\sigma_{2}}\phi\|_{L^{2}}
  \end{equation*}
  and absorbing the last term at the LHS we get \eqref{eq:hardy1}.
  Interpolating between two instances of \eqref{eq:hardy1} with
  two close different values of $\sigma_{1}$ as in the proof of
  Lemma \ref{lem:berndyad}, we get \eqref{eq:hardy2}.
\end{proof}


\end{document}